\documentclass[11pt]{amsart}
\usepackage{enumerate,geometry}           % See geometry.pdf to learn the layout options. There are lots.
\usepackage{amssymb}
\usepackage{epstopdf}
\usepackage{color}
\usepackage{hyperref}

\let\epsilon\varepsilon
\theoremstyle{plain}
\newtheorem{theorem}{Theorem}[section]
\newtheorem{corollary}[theorem]{Corollary}
\newtheorem{lemma}[theorem]{Lemma}

\theoremstyle{definition}

\theoremstyle{remark}

\numberwithin{equation}{section}
\newcommand{\ds}{\displaystyle}
\newcommand{\fn}{\ignorespaces\,}
\newcommand{\half}{\tfrac{1}{2}}

\newcommand{\lip}{\left\langle}
\newcommand{\rip}{\right\rangle}
\newcommand{\lnorm}{\left\Vert}
\newcommand{\rnorm}{\right\Vert}

\newcommand{\Bigglpar}{\Biggl( }
\newcommand{\Biggrpar}{\Biggr) }

\newcommand{\lset}{\left\lbrace}
\newcommand{\rset}{\right\rbrace}
 \newcommand{\labs}{\left\vert}
\newcommand{\rabs}{\right\vert}
\newcommand{\biglabs}{\bigl\vert}
\newcommand{\bigrabs}{\bigr\vert}
\newcommand{\lpar}{\left(}
\newcommand{\rpar}{\right)}

\newcommand{\group}[1]{\mathrm{#1}}

\newcommand{\C}{\mathbb{C}}
\newcommand{\R}{\mathbb{R}}
\newcommand{\T}{\mathbb{T}}

\newcommand{\Hilb}{\mathcal{H}}
\newcommand{\vect}{\mathcal{V}}
\newcommand{\ub}{\mathrm{ub}}

\newcommand{\op}{\mathrm{op}}
\newcommand{\cosec}{\operatorname{cosec}}
\newcommand{\sgn}{\operatorname{sgn}}
\renewcommand{\Re}{\operatorname{Re}}
\renewcommand{\Im}{\operatorname{Im}}

\newcommand\la{\lambda}
\newcommand\Rel{\alpha}
\newcommand\Iml{\beta}

\newcommand\quoz{q}

\newcommand\altfootnote[1]{{}}
\newcommand\footnoteomit[1]{{}}

\newcommand\Hilbab{{\Hilb_{\Rel,a,b}}}
\newcommand\Hilbnorm[2]{\lnorm #1 \rnorm_{\Hilb_{#2}}}

\newcommand\base{{\mathbf{b}}}

\begin{document}
\title{Uniformly bounded representations of $\group{SL}(2,\R)$}
\author{Francesca Astengo}
\address
{Dipartimento di Matematica, Universit\`a di Genova,
16146 Genova, Italia}
\email{astengo@dima.unige.it}

\author{Michael G. Cowling}
\address
{School of Mathematics, University of New South Wales,
UNSW Sydney 2052, Australia}
\email{m.cowling@unsw.edu.au}

\author{Bianca Di Blasio}
\address
{Dipartimento di Matematica e Applicazioni\\
Universit\`a di Milano Bicocca \\
Via Cozzi 53\\
20125 Milano\\ Italia}
\email{bianca.diblasio@unimib.it}

%\date{}                             % Activate to display a given date or no date
\begin{abstract}
We compute the ``norm'' of irreducible uniformly bounded representations of $\group{SL}(2,\R)$.
We show that the Kunze--Stein version of the uniformly bounded representations has minimal norm in the similarity class of uniformly bounded representations.
%We also estimate the norms of many of their matrix coefficients as completely bounded multipliers of the Fourier algebra.
%Our results suggest that the known inequality relating these two norms may not be optimal.
\end{abstract}

\maketitle
\section{Introduction}
We begin by summarising our results briefly.
A representation $\pi$, by which we always mean a continuous representation of a locally compact group $G$ on a Hilbert space $\Hilb_\pi$, is said to be \emph{uniformly bounded} if $\pi(x)$ is a bounded operator on $\Hilb_\pi$ for each $x \in G$, and there is a constant $C$, necessarily no less than $1$, such that
\begin{equation}\label{eq:def-ub}
C^{-1}  \lnorm\xi\rnorm_{\Hilb_\pi}  \leq \lnorm \pi(x) \xi \rnorm_{\Hilb_\pi} \leq C \lnorm\xi\rnorm_{\Hilb_\pi}
\qquad\forall x \in G \quad\forall \xi \in \Hilb_\pi;
\end{equation}
the two inequalities are equivalent because $\pi$ is a representation.
We write $\| \pi(x) \|_{\op}$ for the operator norm of $\pi(x)$ and define the \emph{norm} of $\pi$, written $\| \pi\|_{\ub}$, to be the smallest possible value of $C$ in this inequality.

Suppose that $\pi$ and $\sigma$ are uniformly bounded representations of $G$.
A linear operator from $\Hilb_\pi$ to $\Hilb_\sigma$ such that $\sigma(x) T = T \pi(x)$ for all $x \in G$ is called an intertwiner.
We say that $\pi$ and $\sigma$ are \emph{similar} if there is an intertwiner that is bounded with bounded inverse, and \emph{unitarily equivalent} if there is a unitary intertwiner.
Similarity and unitary equivalence are equivalence relations.
Similar uniformly bounded representations may have different norms and hence not be unitarily equivalent.
In general, little seems to be known about similarity classes of uniformly bounded representations, or about finding uniformly bounded representations in an equivalence class with minimal norm.
Of course, if a uniformly bounded representation is similar to a unitary representation, then the unitary representation has minimal norm in the equivalence class.

In 1955, L.~Ehrenpreis and F.~Mautner \cite{EhrMau1, EhrMau2} showed that $\group{SL}(2, \R)$ has two analytic families of representations $\pi_{\lambda, \epsilon}$, where $\lambda \in \C$ and $\epsilon$ is either $0$ or $1$.  %, realised on Sobolev-type spaces that depend on $\Re(\lambda)$.
These representations have bounded $K$-finite matrix coefficients  (here $K$ is $\group{SO}(2)$) if and only if $|\Re(\lambda)| \leq \half$, and they are uniformly bounded when $| \Re(\lambda) |< \half$; most of them are not similar to unitary representations.
Shortly after, R.A.~Kunze and E.M.~Stein \cite{KunStn1} found a use for these uniformly bounded representations, first realising them on the same Hilbert space, and then using them to prove what is now called the Kunze--Stein phenomenon for $\group{SL}(2,\R)$.

We will define families of Hilbert spaces $\Hilb_{\Rel}$ and $\Hilb_{\Rel,a,b}$, where $\alpha \in (-\frac{1}{2}, \frac{1}{2})$ and $a,b \in \R^+$; the spaces $\Hilb_{\Rel}$ are homogeneous fractional Sobolev spaces, and $\Hilb_{\Rel,a,b}$ and $\Hilb_{\Rel}$ have equivalent norms, so they coincide as spaces of (generalised) functions.
The Kunze--Stein uniformly bounded representations $\pi_{\lambda, \epsilon}$ act on the spaces $\Hilb_{\Rel}$, where $\Rel = \Re\lambda$; the same representation, but with the Hilbert space equipped with the $\Hilb_{\Rel,a,b}$ norm, will be denoted by $\pi_{\lambda, \epsilon, a,b}$.

Our  main theorem is about this family of representations: it shows that the Kunze--Stein representations are optimal, in the sense of having minimal norms, and gives sharp estimates for these norms.

\begin{theorem}
Suppose that $\sigma$ is a uniformly bounded representation of $\group{SL}(2, \R)$ that is similar to $\pi_{\lambda,\epsilon}$, where $| \Re \lambda | < \half$ and $\epsilon \in \{0,1\}$.
Then there exists an equivalent Hilbert norm on $\Hilb_\sigma$ such that $\tau$, the representation $\sigma$ acting on the equivalent Hilbert space, is unitary on the subgroup of lower triangular matrices, and
\[
\|  \sigma \|_{\ub} \geq \| \tau \|_{\ub}.
\]
Further, there exist $a,b \in \R^+$ such that $\tau$ is unitarily equivalent to $\pi_{\lambda, \epsilon, a,b}$ and
\[
\| \tau \|_{\ub} = \| \pi_{\lambda, \epsilon,a,b} \|_{\ub} .
\]
Further, when $a\neq b$ and $(\lambda, \epsilon) \neq (0,1)$,
\[
\| \pi_{\lambda, \epsilon,a,b} \|_{\ub}  > \| \pi_{\lambda, \epsilon} \|_{\ub}
\]
and, when $ |\Im\la | $ is large,
\[
\| \pi_{\lambda, \epsilon} \|_{\ub}
\simeq \frac{(1+|\Im \la|)^{|\Re \la|}}{\half-|\Re \la |}.
\]
\end{theorem}

The expression $A(\la)\simeq B(\la)$ for all $\la$ in a subset $E$ of the domains of $A$ and of $B$ means that there exist (positive) constants $C$ and $C'$ such that
\[
C\, A(\la) \le B(\la) \le C'\,A(\la) \qquad\forall \la\in E.
\]

We now provide more context for our results.
The history of  uniformly bounded representations and their role in harmonic analysis is now quite extensive, and we just outline some of the most important work that we have not already mentioned.

Around 1950, a number of researchers looked at uniformly bounded representations in their studies of amenability.
Once it was known that every uniformly bounded representation of an amenable group is unitarizable, that is, similar to a unitary representation, J.~Dixmier \cite{Dix} asked whether this was true in general or whether this characterized amenability.
As already mentioned, the work of Ehrenpreis and Mautner showed that the former possibility does not hold; the status of the latter is still unresolved.
Considerable effort has gone into the construction of uniformly bounded representations.
Apart from their fundamental paper~\cite{KunStn1}, Kunze and Stein~\cite{KunStn2, KunStn3, KunStn4}, as well as several other authors, constructed analytic families of uniformly bounded representations for many noncompact semisimple Lie groups in the 1960s and 1970s.
In the 1970s and 1980s, uniformly bounded representations were constructed for other groups; for example, A.~Fig\`a-Talamanca and M.A.~Picardello \cite{FigPic} and shortly after T.~Pytlik and R.~Szwarc~\cite{PytSzw} found uniformly bounded representations of the noncommutative free groups.

%In the 1960s, Varopoulos \cite{Var1, Var2, Var3} used tensorial methods to answer questions about thin sets.
%Herz \cite{Her1, Her2}, inspired by Varopoulos, looked at pointwise multipliers of the projective tensor product $L^2(G) \mathbin{\hat{\otimes}} L^2(G)$, and connected these with noncommutative harmonic analysis, and in particular, with certain pointwise multipliers of the Fourier algebra that we will call completely bounded multipliers.
%Bo\.{z}ejko and Fendler (see, e.g., \cite{BozFen1, BozFen2}) developed Herz's ideas, and Haagerup and his collaborators and students demonstrated the central role in harmonic analysis and operator theory of completely bounded multipliers.

Comparatively recently, G.~Pisier~\cite{Pis1, Pis2} has studied uniformly bounded representations, on the one hand taking giant strides towards the solution of the Dixmier similarity problem and on the other developing the links between uniformly bounded representations and multipliers of the Fourier algebra.
Very recently, K.~Juschenko and P.W.~Nowak~\cite{JusNow} linked uniformly bounded representations with the exactness of discrete groups.

%Quite a lot of work has already been done about simple Lie groups of real
%rank one, that is members of the families $\group{SO}(n,1)$, $\group{SU}(n,1)$ and $\group{Sp}(n,1)$, as well as the single exceptional group $F_{4, -20}$.
%For the moment, let us just recall that these groups contain a maximal compact subgroup $K$, and admit Iwasawa decompositions $KAN$, where $A$ is isomorphic to $\R$ and $N$ is nilpotent.
%The \emph{spherical functions} $\phi_\lambda$ are $K$-bi-invariant functions on $G$ (with additional properties) that are important in harmonic analysis on $G$; in particular, they are matrix coefficients of the class-one principal series of representations.
%De Canni\`ere and Haagerup \cite{DcaHaa} and Cowling and Haagerup \cite{CowHaa} estimated the completely bounded multiplier norms of various spherical functions.
%Steenstrup \cite{Stp} computed the norms $\lnorm \phi_\lambda \rnorm_{\cb}$ exactly for the bounded spherical functions on $\group{SO}(n,1)$, and deduced (following the ideas of Haagerup \cite{Haa2}) that there exist completely bounded multipliers of $A(G)$ that do not arise as a matrix coefficient of a uniformly bounded representation, using functional analysis and estimates for various norms associated to spherical functions.
%
In this paper, we return to the roots of all this, and study the uniformly bounded representations of $\group{SL}(2, \R)$ in detail, to  help to further clarify the nature of these still mysterious objects.

Our paper is structured as follows.
In Section~2 we review a few general facts on uniformly bounded representations; in Section~3 we describe the representations of the group $\group{SL}(2,\R)$; and in Section~4 we give precise estimates of the Kunze--Stein representations and prove Theorem~1.1.
%In Section~5 we study the completely bounded multiplier norm of matrix entries of the Kunze--Stein representation, proving Theorem~B.

\section{Background}
We include here a few results about uniformly bounded and unitary representations needed later.

Suppose that $\pi$ is a uniformly bounded representation of a locally  compact group~$G$.
We may produce a new uniformly bounded representation $\rho$ from $\pi$ by putting an equivalent Hilbert norm on the representation space $\Hilb_\pi$.
When we do this, $\rho$ and $\pi$ are similar; indeed, the identity map from $\Hilb_\pi$ with the original norm to $\Hilb_\pi$ with the new norm is a similarity.
In the next lemma we show that when $G$ has a closed amenable subgroup, there is a clever choice for the equivalent norm.

\begin{lemma}\label{lem:good-ub-norm}
Suppose that $G$ is a locally compact group and $H$ is a closed amenable subgroup of $G$.
Suppose also that $\pi$ is a uniformly bounded representation of $G$.
Then there is an equivalent Hilbert space norm on $\Hilb_\pi$ relative to which $H$ acts unitarily.
Further, the norm of $\pi$ relative to the new norm is no greater than that relative to the old norm.
\end{lemma}

\begin{proof}
Take a right invariant mean $m_H$ on $H$.
We define a new inner product on $\Hilb_\pi$ by the formula
\[
\lip \xi, \eta \rip_H = m_H ( h \mapsto \lip \pi(h) \xi, \pi(h) \eta \rip_{\Hilb_\pi}),
\]
and then $\lip \pi(h) \xi, \pi(h) \eta \rip_H = \lip  \xi, \eta \rip_H$ for all $h \in H$ trivially.
From \eqref{eq:def-ub},
\[
\lnorm \pi \rnorm_{\ub}^{-2} \lnorm \xi \rnorm_{\Hilb_\pi}^2 \leq  m_H\lpar h \mapsto  \lnorm \pi(h) \xi \rnorm_{\Hilb_\pi}^2 \rpar \leq \lnorm \pi \rnorm_{\ub}^2 \lnorm \xi \rnorm_{\Hilb_\pi}^2,
\]
and so
\[
\lnorm \pi \rnorm_{\ub}^{-1} \lnorm \xi \rnorm_{\Hilb_\pi} \leq  \lnorm \xi \rnorm_H \leq \lnorm \pi \rnorm_{\ub} \lnorm \xi \rnorm_{\Hilb_\pi};
\]
moreover,
\[
\begin{aligned}
\lnorm \pi(x) \xi \rnorm_H
&= m_H\lpar h \mapsto  \lnorm \pi(h) \pi(x) \xi \rnorm_{\Hilb_\pi}^2 \rpar ^{1/2} \\
&= m_H\lpar h \mapsto  \lnorm \pi(hxh^{-1}) \pi(h) \xi \rnorm_{\Hilb_\pi}^2 \rpar ^{1/2} \\
&\leq m_H\lpar h \mapsto  \lnorm \pi\rnorm_{\ub}^2 \lnorm \pi(h) \xi \rnorm_{\Hilb_\pi}^2 \rpar ^{1/2} \\
&=  \lnorm \pi \rnorm_{\ub}  \lnorm \xi \rnorm_H,
\end{aligned}
\]
as required.
\end{proof}
%\[
%\begin{aligned}
%\lnorm \pi(x) \xi \rnorm_H
%&= m_H\lpar h \mapsto  \lnorm \pi(h) \pi(x) \xi \rnorm_{\Hilb_\pi}^2 \rpar ^{1/2} \\
%&\leq m_H\lpar h \mapsto  \lnorm \pi \rnorm_{\ub}^2  \lnorm \xi \rnorm_{\Hilb_\pi}^2 \rpar ^{1/2} \\
%&\leq \lnorm \pi \rnorm_{\ub} \lnorm \xi \rnorm_{\Hilb_\pi} \\
%&\leq  \lnorm \pi \rnorm_{\ub}^2  \lnorm \xi \rnorm_H.
%\end{aligned}
%\]

The next result is well known,
%(see, for example, \cite[Theorem~4.2.10]{Wolf})
%\footnote{Reference needed. {\color{red} Se ti va bene, abbiamo preso:
% Wolf, Joseph A.
%Harmonic analysis on commutative spaces.
%Mathematical Surveys and Monographs, 142. American Mathematical Society, Providence, RI, 2007.}
%}
but we include a proof for completeness.
It states that similar unitary representations are in fact unitarily equivalent.

\begin{lemma}\label{lemma:corol-of-Schur}
Suppose that $\pi$ and $\sigma$ are irreducible unitary representations of a group $G$, and $T: \Hilb_\pi \to \Hilb_\sigma$ is a bounded operator with bounded inverse that intertwines  $\pi$ and $\sigma$, that is, $\sigma(x) T = T \pi(x)$ for all $x \in G$.
Then there exist $a \in \R^+$ and a unitary map $U: \Hilb_\pi \to \Hilb_\sigma$ such that $T = aU$.
\end{lemma}

\begin{proof}
By taking adjoints, we see that $T^* \sigma(x) = \pi(x) T^*$ for all $x \in G$, and hence $T^*T \pi(x) = \pi(x) T^* T$ for all $x \in G$.
By Schur's lemma, $T^*T$ is a scalar operator; we take $T^*T$ to be multiplication by $a^2$, where $a >0$.
Now $a^{-1}T$ is unitary.
\end{proof}

We are going to use techniques of classical analysis.
We denote by $\|\cdot\|_p$ the usual norm on the Lebesgue space $L^p(\R)$, where $1\le p\le\infty$, and we define the Fourier transform $\hat f$  of a function $f$ on $\R$ by
\[
\hat f (\xi)=\int_\R f(x)\,e^{-ix\xi}\,dx
\qquad \forall \xi \in \R.
\]
Then the Fourier transform extends to a multiple of a unitary operator on $L^2(\R)$, and more precisely,
\[
\|\hat f\|_2=\sqrt{2\pi}\,\|f\|_2
\qquad\forall f \in L^2(\R).
\]

%Similarly, for a function $g$ on $[-\pi, \pi]$,
% we define the Fourier coefficients $\tilde g(m)$ by
%\[
%\tilde g(m) = \frac{1}{2\pi} \int_{-\pi}^{\pi} g(\theta) \,e^{-i m \theta} \, d\theta .
%\]

%We write here some formulas about the gamma function that we shall need later;
%apart from the recurrence relation $\Gamma(z+1)=z\Gamma(z)$, we will use
% the reflection formula (see~\cite[p.~3, formula (6)]{Erd})
% \[
% \Gamma(z)\,\Gamma(1-z)=\pi\,\csc(\pi z),
% \]
% the duplication formula (see~\cite[p.~5, formula (15)]{Erd})
% \[
% \Gamma(z+\half) \, \Gamma(z)=2^{1-2z}\,\pi^{1/2}\,\Gamma(2z),
% \]
%Stirling's asymptotic expansion (see~\cite[p.~47, formula (2)]{Erd})\footnote{This would be the right place to introduce the $\sim$ notation used later, and $\lesssim$ if we are going to use it.}
%\[
%\Gamma(z)=\sqrt{2\pi}\, e^{-z}\, e^{(z-\half)\,\log z}\,
%\left(
%1+O(z^{-1})
%\right)\qquad \text{as  $z\to \infty$ and $|\arg z| \leq \pi - \epsilon$},
%\]
%where $0 \leq \epsilon < \pi$, and Dougall's formula (see \cite[p.~7, formula (1)]{Erd})
%\[
%\sum_{k\in \Z}
%\frac{\Gamma(a+k)\, \Gamma(b+k)}{\Gamma(c+k)\, \Gamma(d+k)}
%=\pi^2\, \csc(\pi a)\, \csc(\pi b)\,
%\frac{\Gamma(c+d-a-b-1)}{
%\Gamma(c-a)\, \Gamma(d-a)\, \Gamma(c-b)\, \Gamma(d-b)} \,,
%\]
%where $a,b \in \C \setminus \Z$, $c,d \in \C$ and $\Re{(a+b-c-d)}<-1$.

\section{The group $\group{SL}(2,\R)$}\label{sec:sl2r}

We now describe $\group{SL}(2, \R)$, abbreviated to $G$ for convenience, and various decompositions and representations thereof.
We present an approach that the second-named author learnt from Kunze many years ago.
First, define subgroups $K$, $M$, $A$, $N$ and $\bar N$ of $G$ as follows:
\[
\begin{gathered}
K = \lset k_\theta : \theta \in \R \rset \qquad M = \lset m_{\pm}  \rset \qquad A = \lset a_s : s \in \R^+ \rset \\ N = \lset n_t : t \in \R \rset \qquad \bar N = \lset \bar n_t : t \in \R \rset,
\end{gathered}
\]
where
\[
%\begin{gathered}
%k_\theta = \begin{pmatrix} \cos\theta & \sin\theta \\ -\sin\theta & \cos\theta \end{pmatrix} = \exp\lpar \theta \begin{pmatrix} 0 & 1 \\ -1 & 0 \end{pmatrix} \rpar
%\\
%m_\pm = \begin{pmatrix} \pm 1 &  0 \\  0 & \pm 1 \end{pmatrix}
%\qquad
%a_s = \begin{pmatrix} s &0 \\ 0  & s^{-1} \end{pmatrix} =  \exp\lpar \log(s)  \begin{pmatrix} 1 & 0\\ 0 & -1 \end{pmatrix} \rpar
%\\
%n_t = \begin{pmatrix} 1 & 0 \\ t & 1 \end{pmatrix} = \exp\lpar t \begin{pmatrix} 0 & 0\\ 1 & 0 \end{pmatrix} \rpar
%\qquad
%\bar n_t = \begin{pmatrix} 1 & t \\ 0 & 1 \end{pmatrix} = \exp\lpar t \begin{pmatrix} 0 & 1\\ 0 & 0 \end{pmatrix} \rpar;
%\end{gathered}
\begin{gathered}
k_\theta = \begin{pmatrix} \cos\theta & \sin\theta \\ -\sin\theta & \cos\theta \end{pmatrix} 
\qquad
m_\pm = \begin{pmatrix} \pm 1 &  0 \\  0 & \pm 1 \end{pmatrix}
\qquad
a_s = \begin{pmatrix} s &0 \\ 0  & s^{-1} \end{pmatrix} 
\\
n_t = \begin{pmatrix} 1 & 0 \\ t & 1 \end{pmatrix} 
\qquad
\bar n_t = \begin{pmatrix} 1 & t \\ 0 & 1 \end{pmatrix} ;
\end{gathered}
\]
we will write $w$ for the rotation $k_{\pi/2}$.

Consider $\R^2$ as a space of row vectors, and $G$ acting on $\R^2$ by right multiplication.
Then $G$ fixes the origin and acts transitively on $\R^2 \setminus \{(0,0)\}$.
Write $\base$ for the ``base point'' $(1,0)$, and $B$ for the space $\R^2\setminus\{(0,0)\}$.
The subgroup $N$ is the stabiliser of the point $\base$, and so $B$ may be identified with the coset space $N\backslash G$.
The polar decomposition in $B$ leads to the Iwasawa decomposition of $G$: every element $x$ of $G$ may be expressed uniquely in the form
\[
x = nak
\]
where $n \in N$, $a \in A$ and $k \in K$.
Indeed, $\base x \in B$, and if we choose $s = \| \base x \|$ and $\theta = \arg(\base x)$, then $\base  a_s k_\theta = \base x$; there is therefore an element $n$ of $N$ such that $x = n a_s k_\theta$; further, since $k_\theta$ and $a_s$ are uniquely determined, $n$ is also unique.
We may describe the Bruhat decomposition in similar terms: $B$ is the disjoint union of the real axis (minus the origin) and $\R^2$ minus the real axis, and this corresponds to writing $G$ as the disjoint union $(NAM) \sqcup (NAM w NAM)$.
%The third frequently used decomposition, the Cartan decomposition, does not have an obvious interpretation in terms of $B$.
%It states that every element of $G$ may be written in the form $k_\theta a_s k_\phi$, where $s \geq 1$; if $x \in K$ then $s = 1$ and there are many choices for $k_\theta$ and $k_\phi$; otherwise, the decomposition is unique up to changes in $\theta$ and $\phi$ by adding (or subtracting) $\pi$ to both or $2\pi$ to either.
%This result is derived using the polar decomposition of a matrix.

We now consider the space $\vect_{\lambda,\epsilon}$, where
$\lambda \in \C$ and $\epsilon$ is either $0$ or $1$, of smooth functions on $B$ that satisfy
%\footnote{rispetto alla versione precedente, qui usato $\lambda=(z+1)/2\quad$ $z=2\lambda-1$}
\[
f(\delta v) = |\delta|^{2\lambda-1} \sgn(\delta)^\epsilon f(v)
\qquad\forall v \in B \quad\forall \delta \in \R\setminus\{0\} ,
\]
equipped with the topology of locally uniform convergence of all partial derivatives.
%These functions may also be considered as sections of a line bundle over $G/NA$ or $G/NAM$.
Since $G$ acts on $B$ and commutes with scalar multiplication, $G$ acts on $\vect_{\lambda,\epsilon}$ by the formula
\[
\pi_{\lambda,\epsilon}(x) f(v) = f(vx)
\qquad\forall v \in B\quad\forall x \in G.
\]
We obtain the ``compact picture'' of the representation by restricting $v$ to lie in the circle $\base K$, and observing that
\[
\pi_{\lambda,\epsilon}(x) f(v) = \labs vx \rabs^{2\lambda-1} f( \labs vx \rabs^{-1} vx) .
\]
The ``noncompact picture'' is obtained similarly, by restricting $v$ to lie on the line $\base \bar N$, and observing that
%\footnote{Ho tolto la decomposizione di Cartan, perch\'e non la usiamo.
%Avevamo due  $b$ vicino con diverso significato: $b=$ base point e $b$ in $\begin{pmatrix} a & b \\ c& d \end{pmatrix}$. Dopo usiamo anche $b$ nella norma $\Hilbab$.
%Ho aggiunto la macro $\backslash$base che ora produce $\base$. Basta cambiare quella. }
\begin{equation}\label{eq:noncompact}
\pi_{\lambda,\epsilon}(x) f(1,t) = \sgn^\epsilon (a+tc) \labs a+tc \rabs^{2\lambda-1} f( 1, x \cdot t) \qquad\forall t \in \R,
\end{equation}
where
\[
x = \begin{pmatrix} a & b \\ c& d \end{pmatrix}
\qquad\text{and}\qquad
x \cdot t = \dfrac{b+dt}{a+ct}.
\]
Clearly some care is required ``at infinity'' in this version of the representation.

By completing $\vect_{\lambda,\epsilon}$ in an appropriate norm, we may obtain representations of $G$ on Hilbert or Banach spaces.
First note that $|f|$ is even, and so
\[
\int_{-\pi/2}^{\pi/2} |f(\cos\theta,\sin\theta )| \,d\theta = \frac{1}{2} \int_{-\pi}^{\pi} |f(\cos\theta,\sin\theta )| \,d\theta .
\]
Now observe that if $p(\Re\lambda-\frac12) = 1$, then
\[
\int_\R \labs f(\base  \bar n_t) \rabs^p \,dt  = \int_\R \labs f(1,t) \rabs^p \,dt  = \int_{-\pi/2}^{\pi/2} \labs f(\cos\theta,\sin\theta ) \rabs^p \,d\theta = \int_{-\pi/2}^{\pi/2} \labs f(\base  k_\theta ) \rabs^p \,d\theta.
\]
Indeed, $f(1, \tan\theta) = (1+\tan^2\theta)^{\lambda-1/2} \,f(\cos\theta,\sin\theta)$, and the formula above is a consequence of setting $t = \tan\theta$ and changing variables.
Taking $p$th roots, we see that
\[
\lpar \int_\R \labs f(\base  \bar n_t) \rabs^p \,dt \rpar^{1/p}  = \lpar \frac{1}{2} \int_{-\pi}^{\pi} \labs f(\base  k_\theta ) \rabs^p \,d\theta \rpar^{1/p}.
\]
Now the left hand integral is trivially unchanged if we replace $f$ by $\pi_{\lambda,\epsilon}(\bar n_u)f$, while the right hand integral is trivially unchanged if we replace $f$ by $\pi_{\lambda,\epsilon}(k_\phi)f$.
Since the smallest subgroup of $G$ that contains both $\bar N$ and $K$ is $G$, it follows that in fact the integrals above are unchanged if we replace $f$ by $\pi_{\lambda,\epsilon}(x)f$ for any $x \in G$.
%In bundle language, we have shown that $\labs f \rabs^p$ is a density.

The representations $\pi_{\lambda,\epsilon}$, where $\Re \lambda = 0$,
%\altfootnote{$\Re z = -1$}
are isometric on a Hilbert space, and hence unitary.
%{\color{green}
These representations are irreducible, except when $(\lambda,\epsilon) = (0,1)$; they are known as the unitary principal series.
%}

Take $f \in \vect_{\lambda,\epsilon}$ and $g \in \vect_{\mu,\epsilon}$, where $\lambda+\mu=0$.
\altfootnote{$z+w=-2$}
By a simple variant of the argument above, we may show that
\[
\int_\R  f(1,t) \, g(1,t) \,dt = \int_{-\pi/2}^{\pi/2}  f(\cos\theta,\sin\theta ) \, g(\cos\theta,\sin\theta ) \,d\theta .
\]
We define $(f, g)$ to be either of the above integrals, then the bilinear form $( \cdot,\cdot)$  is well defined and $G$-invariant, and exhibits the canonical duality between $\vect_{\lambda,\epsilon}$ and $\vect_{-\lambda,\epsilon}$.

When $\lambda \in (-1/2,0)$, we may find a Hilbert norm such that $\pi_{\lambda,0}$ acts isometrically, and hence unitarily.
Since
\[
( \tan\theta - \tan\phi ) \cos\theta \cos\phi %= \sin\theta \cos\phi - \sin\phi \cos\theta
= \sin(\theta - \phi),
\]
it follows that if $\lambda \in (-1/2,0)$ and $f \in \vect_{\lambda,0}$, then
\[
\begin{aligned}{}
&\int_{\R} \int_{\R} f(1,t) \, \bar f(1,u) \labs t - u\rabs ^{-(1+2\lambda)} \,dt \,du \\
&\quad = \int_{-\pi/2}^{\pi/2}  \int_{-\pi/2}^{\pi/2} f(1,\tan\theta) \, \bar f(1, \tan\phi) \,\frac{ \labs \tan\theta - \tan\phi \rabs ^{-(1+2\lambda)}} {\cos^2\theta \cos^2\phi} \,d\theta\,d\phi \\
&\quad =  \int_{-\pi/2}^{\pi/2}  \int_{-\pi/2}^{\pi/2} f(\cos\theta,\sin\theta ) \,\bar f(\cos\phi,\sin\phi ) \,\frac{\labs \tan\theta - \tan\phi \rabs ^{-(1+2\lambda)}}{ \cos^{1+2\lambda} \theta \cos^{1+2\lambda} \phi} \,d\theta\,d\phi \\
&\quad =  \int_{-\pi/2}^{\pi/2}  \int_{-\pi/2}^{\pi/2} f(\cos\theta,\sin\theta ) \,\bar f(\cos\phi,\sin\phi ) \labs \sin(\theta-\phi) \rabs ^{-(1+2\lambda)} \,d\theta\,d\phi \\
&\quad =  \frac{1}{4}\int_{-\pi}^{\pi}  \int_{-\pi}^{\pi} f(\cos\theta,\sin\theta ) \,\bar f(\cos\phi,\sin\phi ) \labs \cosec(\theta-\phi) \rabs ^{1+2\lambda} \,d\theta\,d\phi ;
\end{aligned}
\]%
\altfootnote{
\[
\begin{aligned}{}
&\int_{\R} \int_{\R} f(1,t) \, \bar f(1,u) \labs t - u\rabs ^{-2-z} \,dt \,du \\
&\quad = \int_{-\pi/2}^{\pi/2}  \int_{-\pi/2}^{\pi/2} f(1,\tan\theta) \, \bar f(1, \tan\phi) \labs \tan\theta - \tan\phi \rabs ^{-2-z} \sec^2\theta \sec^2\phi \,d\theta\,d\phi \\
&\quad =  \int_{-\pi/2}^{\pi/2}  \int_{-\pi/2}^{\pi/2} f(\cos\theta,\sin\theta ) \,\bar f(\cos\phi,\sin\phi ) \labs \tan\theta - \tan\phi \rabs ^{-2-z} \sec^{2+z} \theta \sec^{2+z} \phi \,d\theta\,d\phi \\
&\quad =  \int_{-\pi/2}^{\pi/2}  \int_{-\pi/2}^{\pi/2} f(\cos\theta,\sin\theta ) \,\bar f(\cos\phi,\sin\phi ) \labs \sin(\theta-\phi) \rabs ^{-2-z} \,d\theta\,d\phi \\
&\quad =  \frac{1}{4}\int_{-\pi}^{\pi}  \int_{-\pi}^{\pi} f(\cos\theta,\sin\theta ) \,\bar f(\cos\phi,\sin\phi ) \labs \cosec(\theta-\phi) \rabs ^{2+z} \,d\theta\,d\phi ;
\end{aligned}
\]}%
all integrals converge absolutely.
Taking square roots, and introducing some notation,  we see that
\[
\begin{aligned}{}
\Hilbnorm{f}{\lambda, \R}
& := C_\lambda \lpar \int_{\R} \int_{\R} f(1,t) \, \bar f(1,u) \labs t - u\rabs ^{-(1+2\lambda)} \,dt \,du \rpar^{1/2}  \\
&=  C_\lambda \lpar \frac{1}{4}\int_{-\pi}^{\pi}  \int_{-\pi}^{\pi} f(\cos\theta,\sin\theta ) \,\bar f(\cos\phi,\sin\phi ) \labs \cosec(\theta-\phi) \rabs ^{1+2\lambda} \,d\theta\,d\phi \rpar^{1/2} \\
&=:  \Hilbnorm{f}{\lambda, \T} ,
\end{aligned}
\]
say; the positive constant $C_\lambda$ is chosen such that
\begin{equation}\label{eq:Hilbnorm}
\Hilbnorm{f}{\lambda,\R}
= \lpar \int_{\R}  \biglabs \labs r \rabs^{\la} \hat f(1,r) \bigrabs^2  \,dr \rpar^{1/2} ,
\end{equation}
where the Fourier transform acts in the second variable only.
Thus the first norm is a homogeneous fractional Sobolev norm on $\R$; the second norm is equivalent to an inhomogeneous fractional Sobolev norm on even functions on the circle.
Much as before, the norms above are $G$-invariant, and so $\pi_{\lambda,0}$ acts unitarily on the completion of $\vect_{\lambda,0}$ in this norm, which may be identified with a space of distributions on $B$.
The duality described in the preceding paragraph enables us to find a Hilbert norm so that $\pi_{\lambda,0}$ acts unitarily when $\lambda \in (0,1/2)$; this norm is also given by the formula \eqref{eq:Hilbnorm}.
The representations $\pi_{\lambda,0}$ and $\pi_{-\lambda,0}$ are unitarily equivalent.
\altfootnote{$z \in (-1,0)$}
The family of representations $\pi_{\lambda,0}$, where $\lambda \in (-\half,0) \cup (0, \half) $ is called the complementary series.

Recall that when $\lambda$ is purely imaginary, the representations $\pi_{\lambda, \epsilon}$ act unitarily on the completion of $\vect_{\lambda,0}$ in the norm $\Hilbnorm{\cdot}{0, \R}$, giving us the unitary principal series.
For completeness, we mention that $G$ has some additional irreducible unitary representations, namely, the trivial representation, the discrete series of representations (which appear as subspaces or quotient spaces of the representations $\pi_{\lambda,\epsilon}$ when $\epsilon = 0$ and $\lambda = \pm \half, \pm \frac{3}{2}, \dots$ and when $\epsilon = 1$ and $\lambda = \pm 1, \pm2$, \dots), and the limits of discrete series representations, which are the two distinct summands of the reducible representation $\pi_{0,1}$.
%}

\section{Uniformly bounded representations}

Producing uniformly bounded representations is more difficult than producing unitary representations.
In this section, we do this, and compute the norms of the Kunze--Stein representations and of more general uniformly bounded representations.
We denote the Cartesian form of a complex number $\la$ by $\Rel+i\Iml$; we suppose throughout that $|\Rel | < \half$.

Since our analysis is carried out in the noncompact picture, we shall more simply write $f(t)$ instead of $f(1,t)$. 
Moreover we shall realise the representation $\pi_{\la,\epsilon}$ on the completion $\Hilb_\Rel$ of the space $C_c^\infty(\R)$ of smooth functions with compact support in the norm
\[
\Hilbnorm{f}{\Rel}=
\lpar \int_{0}^{\infty}  \biglabs \labs r \rabs^{\Rel} \hat f(r) \bigrabs^2  \,dr + \int_{-\infty}^0  \biglabs \labs r \rabs^{\Rel} \hat f(r) \bigrabs^2 \,dr \rpar^{1/2} .
\]
The following lemma, observed by Kunze and Stein \cite{KunStn1}, is one of the keys to our approach.

\begin{lemma}\label{lem:KS1}
Suppose that $|\Rel| < \half$.
Then $\pi_{\la,\epsilon}$ is a uniformly bounded representation of the group $G$ on $\Hilb_\Rel$, and
\[
\lnorm \pi_{\lambda, \epsilon} \rnorm_{\ub} = \lnorm \pi_{\lambda, \epsilon}(w) \rnorm_{\op} .
\]
\end{lemma}

\begin{proof}
It is easy to see that the subgroup $\bar NAM$ of $G$ acts unitarily on $\Hilb_\Rel $.
Now the Bruhat decomposition  $G = \bar N AM \sqcup \bar N AM w \bar N AM$ implies that $G$ acts uniformly boundedly on this space if and only if $\pi_{\lambda,\epsilon}(w)$ is bounded thereon, and the uniformly bounded norm of the representation is the norm of the single operator $\pi_{\lambda,\epsilon}(w)$.
%For more details, see \cite{KunStn1}.
\end{proof}

Kunze and Stein \cite{KunStn1} just estimated the operator norm $\lnorm \pi_{\lambda, \epsilon}(w) \rnorm_{\op}$; in this paper we compute it exactly.
First however we construct some more general uniformly bounded representations.

Observe that for all $f$ in $C_c^\infty(\R)$,
\begin{equation}
\label{eq:weyl-action}
\pi_{\lambda,\epsilon}(w) f (t)
%= \sgn^\epsilon(-t) \labs t\rabs^{2\lambda-1} f(-1/t)
 =\sgn^\epsilon(-t) \labs t\rabs^{2(\Rel+i\Iml)-1}
 f(-1/t)
 \quad \forall t\in \R.
\end{equation}

We may produce a new uniformly bounded representation by defining an equivalent norm $\lnorm \cdot \rnorm_{\Hilbab}$ on $\Hilb_\Rel$, thus:
\[
\lnorm f \rnorm_{\Hilbab} =  \lpar a \int_{0}^{\infty}  \biglabs \labs r \rabs^{\Rel} \hat f(r) \bigrabs^2  \,dr + b \int_{-\infty}^0  \biglabs \labs r \rabs^{\Rel} \hat f(r) \bigrabs^2 \,dr \rpar^{1/2} ,
\]
where $a,b>0$, and considering $\pi_{\lambda, \epsilon}$ acting on $\Hilb_\Rel$ with this new norm.
This representation is uniformly bounded because the new norm is equivalent to the old norm.
Further, when $a = b$, the norm on the space $\Hilbab$ is a multiple of the norm on $\Hilb_\Rel$; the uniformly bounded norms of $\pi_{\lambda, \epsilon}$ on the spaces $\Hilbab$ and $\Hilb_\Rel$ therefore coincide in this case.

We write $\pi_{\la,\epsilon,a,b}$ for the representation $\pi_{\lambda, \epsilon}$ on $\Hilbab$.

\begin{lemma}\label{sigma}
Suppose that $\sigma$ is a uniformly bounded representation of $G$ that is similar to $\pi_{\la,\epsilon}$.
Then there exist $\tau$, obtained from $\sigma$ by renorming the representation space, and $a$ and $b$ in $\R^+$ such that $\tau$ and $\pi_{\la,\epsilon,a,b} $ are unitarily equivalent.
%
%Then there exist an equivalent Hilbert norm on $\Hilb_\sigma$ and $a$ and $b$ in $\R^+$ and a unitary map $U: \Hilbab \to \Hilb_\sigma$ such that $U^{-1} \sigma U = \pi_{\la,\epsilon,a,b}$.
Further,
\[
\| \sigma \|_{\ub} \geq \| \tau \|_{\ub} =\| \pi_{\la,\epsilon,a,b} \|_{\ub} .
\]
%, the norm of $\pi_{\la,\epsilon}$ acting on $ \Hilb_{a,b}$.
\end{lemma}

\begin{proof}
Take a uniformly bounded representation $\sigma$ of $G$, similar to $\pi_{\la,\epsilon}$.
Since $MA{\bar N}$ is amenable, by Lemma \ref{lem:good-ub-norm} there exists an equivalent Hilbert norm on $\Hilb_\sigma$ such that $\tau$, the representation $\sigma$ acting on the equivalent Hilbert space, is unitary on $MA{\bar N}$.
Moreover $\| \sigma \|_{\ub} \geq \| \tau \|_{\ub}$.
%Since $MA{\bar N}$ is amenable, we may assume that $\sigma |_{MA{\bar N}}$ is unitary,
%by Lemma \ref{lem:good-ub-norm}.

Since also $\tau$ and $\pi_{\la,\epsilon}$ are similar,
there exists a bounded map $T: \Hilb_\Rel \to \Hilb_{\tau}$ with bounded inverse such that $T \pi_{\la,\epsilon}(x) = \tau(x) T$ for all $x \in G$.

The representation space $\Hilb_\Rel$ splits into two complementary unitarily inequivalent  subspaces, $\Hilb^\pm_\Rel$ say, given by
\[
\Hilb^\pm_\Rel
=\{f\in \Hilb_\Rel\, :\, \hat f|_{\R^\mp}=0\} ,
\]
on both of which $\pi_{\la,\epsilon}|_{MA{\bar N}}$ acts irreducibly.
We define $\Hilb_{\tau}^\pm = T \Hilb^\pm_\Rel$; then the unitary representation ${\tau}|_{MA{\bar N}}$ acts irreducibly on $\Hilb_{\tau}^+$ and $\Hilb_{\tau}^-$, whence $\Hilb_{{\tau}} = \Hilb^+_{\tau} \oplus \Hilb^-_{\tau}$.
By Lemma \ref{lemma:corol-of-Schur} applied to each irreducible component,  $T |_{\Hilb^\pm_\Rel}:\Hilb^\pm_\Rel \to \Hilb_{\tau}^\pm$ is a multiple of a unitary map, and so for the right choice of $a$ and $b$, the intertwining operator is unitary from $\Hilbab$ to $\Hilb_{\tau}$, and we are done.
\end{proof}

The rest of this section will be devoted to computing $  \| \pi_{\la,\epsilon,a,b} \|_{\ub} $, where $a,b \in \R^+$.
When $a=b$ the norm $  \| \pi_{\la,\epsilon,a,b} \|_{\ub} $ reduces to the norm of the Kunze--Stein representation~$\pi_{\la,\epsilon}$ and we will simply write  $  \| \pi_{\la,\epsilon} \|_{\ub} $.

\begin{lemma}\label{moltiplicatore}
%Let $\sigma$ be a uniformly bounded representation of $G$ that is similar to $\pi_{\la,\epsilon}$,  and
Suppose that $a,b \in \R^+$.
Define $\theta = (a-b)/(a+b)$,
\[
 m_{0,\Rel}(u)=\sqrt{2\pi}\,\, 2^{\Rel +iu} \frac{\Gamma\left(\frac14+\frac{\Rel+iu}2\right)}{\Gamma\left(\frac14-\frac{\Rel+iu}2\right)}
\qquad
 m_{1,\Rel}(u)=i\sqrt{2\pi}\,\, 2^{\Rel +iu} \frac{\Gamma\left(\frac34+\frac{\Rel+iu}2\right)}{\Gamma\left(\frac34-\frac{\Rel+iu}2\right)}  \,,
\]
and
\[
\quoz_{0,\epsilon}=\quoz_{0,\epsilon,\Rel+i\Iml}=\frac{m_{0,\Rel} (2\Iml - \cdot)}{m_{\epsilon,\Rel}}
\qquad \qquad
\quoz_{1,\epsilon}=\quoz_{1,\epsilon,\Rel+i\Iml}=-\frac{m_{1,\Rel} (2\Iml - \cdot)}{m_{1-\epsilon,\Rel}}  \,.
\]
Then $\| \pi_{\la,\epsilon,a,b} \|_{\ub}^2$ is equal to
 \begin{equation}\label{norma_sigma}
\sup\left\{
%\left(
\frac{
 \lnorm
h_\epsilon \,\quoz_{0,\epsilon}
 \rnorm_2^2
+ \lnorm
h_{1-\epsilon} \,\quoz_{1,\epsilon}
 \rnorm_2^2
+2\theta
\Re \langle
h_\epsilon \, \quoz_{0,\epsilon} \, ,
h_{1-\epsilon} \, \quoz_{1,\epsilon}
\rangle_{L^2}
}{ \lnorm h_0
 \rnorm_2^2
+ \lnorm h_1
 \rnorm_2^2
+2\theta
\Re \langle h_0 \, , h_1
\rangle_{L^2}
}
%\right)^{1/2}
\,:\,
%(h_0,h_1)\in L^2(\R)\oplus L^2(\R)\setminus\set{0}
\|h_0\|_2^2+\|h_1\|_2^2\neq 0
\right\}.
\end{equation}
 \end{lemma}

For brevity, when the dependence on the parameter $\Rel$ is not important,
we shall omit it in subscripts, for example, we shall simply write $m_0$ instead of $m_{0,\Rel}$.

\begin{proof}
First, we need an efficient way to compute $\|f\|_{\Hilbab}$.
We will decompose $f$ into its even and odd parts and use the Mellin transform.
We write
\[
f(t) = \int_{\R} c_0(u) \labs t\rabs^{\Rel -\half +iu} \,du + \int_{\R} c_1(u) \sgn(t) \labs t \rabs^{\Rel -\half +iu} \,du.
\]
As proved in~\cite[p.~160]{StnWei} and \cite[p.~173 formulae (12) and (13)]{GelShi}, the Fourier transform of
$\labs \cdot \rabs^{\Rel -\half +iu} $ is given by $m_0(u) \labs\cdot\rabs^{ -\Rel-\half  - iu} $,
where
\[
m_0(u)=m_{0,\Rel}(u)=\sqrt{2\pi}\,\, 2^{\Rel +iu} \frac{\Gamma\left(\frac14+\frac{\Rel+iu}2\right)}{\Gamma\left(\frac14-\frac{\Rel+iu}2\right)},
\]
and the Fourier transform of $\sgn(\cdot)\labs \cdot \rabs^{\Rel -\half +iu} $ is given by
$m_1(u) \labs\cdot\rabs^{ -\Rel-\half  - iu} \,\sgn(\cdot)$,
where
\[
m_1(u)=m_{1,\Rel}(u)=i\sqrt{2\pi}\,\, 2^{\Rel +iu} \frac{\Gamma\left(\frac34+\frac{\Rel+iu}2\right)}{\Gamma\left(\frac34-\frac{\Rel+iu}2\right)}.
\]
Therefore
%\[
%\begin{aligned}
%\hat f(\xi)
%&= \int_{\R} c_0(u) \fn m_0(u) \labs\xi\rabs^{ -\Rel-\half  - iu} \,du + \int_{\R} c_1(u) \fn m_1(u) \sgn(\xi) \labs\xi\rabs^{ -\Rel-\half  - iu} \,du \\
%&= \int_{\R} ( c_0(u) \fn m_0(u) + \sgn(\xi) \fn c_1(u) \, m_1(u) ) \labs\xi\rabs^{ -\Rel-\half  - iu} \,du ,
%\end{aligned}
%\]
%and
\[
\begin{aligned}
\labs\xi\rabs^{\Rel} \hat f(\xi)
&= \labs\xi\rabs^{ -\half} \int_{\R} ( c_0(u) \fn m_0(u) + \sgn(\xi) \fn c_1(u) \, m_1(u) ) \labs\xi\rabs^{- iu} \,du .
\end{aligned}
\]

In light of the definition of $\lnorm \cdot \rnorm_{\Hilbab}$ above and  the Plancherel Theorem,
\[
\begin{aligned}
 \lnorm f \rnorm_{\Hilbab} ^2
 &=  a \int_0^{+\infty} \labs \int_{\R} ( c_0(u) \fn m_0(u) + \fn c_1(u) \, m_1(u) ) \labs\xi\rabs^{- iu} \,du \rabs ^2 \,\frac{d\xi}{\labs \xi\rabs} \\
 &\qquad
            + b \int_{-\infty}^0 \labs \int_{\R} ( c_0(u) \fn m_0(u) - \fn c_1(u) \, m_1(u) ) \labs\xi\rabs^{- iu} \,du \rabs ^2 \,\frac{d\xi}{\labs \xi\rabs} \\
 %&=  a \int_0^{+\infty} \labs \int_{\R} ( c_0(u) \fn m_0(u) + \fn c_1(u) \, m_1(u) ) \xi^{- iu} \,du \rabs ^2 \,\frac{d\xi}{ \xi} \\
 %&\qquad
 %           + b \int_0^{\infty} \labs \int_{\R} ( c_0(u) \fn m_0(u) - \fn c_1(u) \, m_1(u) ) \xi^{- iu} \,du \rabs ^2 \,\frac{d\xi}{\xi} \\
 &=  a \int_{\R} \labs \int_{\R} ( c_0(u) \fn m_0(u) + \fn c_1(u) \, m_1(u) ) \,e^{- itu} \,du \rabs ^2 \,dt \\
 &\qquad
            + b \int_{\R} \labs \int_{\R} ( c_0(u) \fn m_0(u) - \fn c_1(u) \, m_1(u) ) \,e^{- itu} \,du \rabs ^2 \,dt \\
&= 2\pi\Bigglpar  a \int_{\R} \labs c_0(u) \fn m_0(u) + \fn c_1(u) \, m_1(u) \rabs ^2 \,du
\\
&\qquad
            + b \int_{\R} \labs c_0(u) \fn m_0(u) - \fn c_1(u) \, m_1(u) ) \rabs ^2 \,du \Biggrpar \\
&= 2\pi\Bigglpar  (a+b) \int_{\R} \labs c_0(u) \fn m_0(u) \rabs^2 + \labs c_1(u) \, m_1(u) \rabs ^2 \,du \\
&\qquad
  + 2(a-b) \int_{\R} \Re\lpar c_0(u) \fn m_0(u) \fn \bar c_1(u) \fn \bar m_1(u) ) \rpar \,du \Biggrpar .
\\
&= 2\pi (a+b) \Bigglpar
 \|c_0  \fn m_0 \|_2^2+\|c_1  \fn  m_1 \|_2^2
 +2\, \theta\, \Re \langle c_0  \fn m_0, c_1  \fn  m_1\rangle_{L^2}  \Biggrpar ,
\end{aligned}
\]
where $\theta = (a-b)/(a+b)$. 
Since $a$ and $b$ are positive, $-1 < \theta < 1$.

As in the proof of Lemma \ref{lem:KS1}, since $\bar NAM$ acts unitarily on $\Hilbab$,
\[
\| \pi_{\la, \epsilon,a,b} \|_{\ub} = \| \pi_{\la,\epsilon,a,b}(w) \|_{\op}.
\]
Suppose  that
\begin{align*}
f(t)
&= \int_{\R} c_0(u)\,\labs t \rabs^{\Rel -\half +iu} \,du + \int_{\R} c_1(u) \,\sgn(t) \, \labs t \rabs^{\Rel -\half +iu} \,du \\
&= \int_{\R} [ c_0(u) + c_1(u) \sgn(t)  ] \labs t \rabs^{\Rel -\half +iu} \,du.
\end{align*}
%\footnote{\color{red} Ho tolto  $2i\Iml$ Ok per me, ne ho tolto un altro}
Then, from \eqref{eq:weyl-action} and linearity,
\[
\begin{aligned}
{} [\pi_{\la,\epsilon,a,b}(w) f](t)
&=  \int_{\R}  [ c_0(u) + c_1(u) \sgn(-1/t)  ] \sgn^\epsilon (-t) \, \labs t \rabs^{2\lambda - 1} \labs -1/ t \rabs^{\Rel -\half +iu} \,du  \\
&=  (-1)^\epsilon \int_{\R}  [ c_0(u) - c_1(u) \sgn(t) ] \sgn^\epsilon (t) \,\labs t \rabs^{2\lambda - 1 - (\Rel -\half +iu)} \,du  \\
&=  (-1)^\epsilon \int_{\R}  [ c_0(u) - c_1(u) \sgn(t) ] \sgn^\epsilon (t) \,\labs t \rabs^{\Rel -\half + i(2\beta - u)} \,du  \\
&=  (-1)^\epsilon \int_{\R}  [ c_0(2\beta - u) - c_1(2\beta - u) \sgn(t) ] \sgn^\epsilon (t) \,\labs t \rabs^{\Rel -\half + iu} \,du \\
&=  \int_{\R} [ c_\epsilon(2\beta - u) - c_{1-\epsilon} (2\beta - u) \sgn (t)  ] \labs t \rabs^{\Rel -\half + iu} \,du .
\end{aligned}
\]

%\footnote{\color{red} Nell'ultima riga ho tolto $+ (-1)^{1-\epsilon} $ e messo un segno negativo: controllare.
%Ok anche per me. Poi
%cambiano i segni di conseguenza, anche nell'enunciato cambia un segno ad una delle "q"}
%In light of the previous lemma, we need only compute this operator norm.

Hence, when $f\neq 0$,
\[
\begin{aligned}
&
\frac{\lnorm \pi_{\la,\epsilon,a,b}(w) f \rnorm_{\Hilbab}^2}{\lnorm f \rnorm_{\Hilbab}^2}
\\
&\quad
=\frac{
\|c_\epsilon(2\Iml - \cdot)  \fn m_0 \|_2^2+\|c_{1-\epsilon}(2\Iml - \cdot)  \fn  m_1 \|_2^2
 -2\,  \theta\, \Re \langle c_\epsilon(2\Iml - \cdot)  \fn m_0, c_{1-\epsilon}(2\Iml - \cdot)  \fn  m_1\rangle_{L^2}
}
{
\|c_0  \fn m_0 \|_2^2+\|c_1  \fn  m_1 \|_2^2
 +2\, \theta\, \Re \langle c_0  \fn m_0, c_1  \fn  m_1\rangle_{L^2}
}
\\
&\quad
=\frac{
\| m_0 (2\Iml - \cdot)\fn c_\epsilon \|_2^2+\|   m_1 (2\Iml - \cdot)\fn c_{1-\epsilon} \|_2^2
 -2 \,\theta\,
 \Re \langle
 m_0(2\Iml - \cdot)\fn c_\epsilon
 ,
 m_1(2\Iml - \cdot)\fn c_{1-\epsilon}
 \rangle_{L^2}
}
{
\|c_0  \fn m_0 \|_2^2+\|c_1  \fn  m_1 \|_2^2
 +2\, \theta\, \Re \langle c_0  \fn m_0, c_1  \fn  m_1\rangle_{L^2}
} \,.
\end{aligned}
\]

To find the norm of  the operator $\pi_{\la,\epsilon,a,b}(w)$,
we need to take the supremum of the last expression as $f$
varies over $\Hilbab$.
The denominator of the last expression is the square of the $\Hilbab$-norm of $f$, so to take an arbitrary $f$ in $\Hilbab$, we may replace the function $c_km_k$ by $h_k$, and take arbitrary $h_k \in L^2(\R)$.
Hence
\[
\begin{aligned}
&
\phantom{=}
\sup\left\{
\frac{\lnorm \pi_{\la,\epsilon}(w) f \rnorm_{\Hilbab}^2
}{
\lnorm f \rnorm_{\Hilbab}^2}
\, :\,f\in \Hilbab\,, f\neq 0\right\}
\\
%&\frac{
% \lnorm
%c_\epsilon (\cdot-2\Iml)\,m_0
% \rnorm_2^2
%+ \lnorm
%c_{1-\epsilon} (\cdot-2\Iml)\,m_1
% \rnorm_2^2
%+2\theta
%\langle c_\epsilon (\cdot-2\Iml)\,m_0 \, , c_{1-\epsilon} (\cdot-2\Iml)\,m_1
%\rangle
%}{ \lnorm h_0
% \rnorm_2^2
%+ \lnorm h_1
% \rnorm_2^2
%+2\theta
%\langle h_0 \, , h_1
%\rangle
%}
%\\
%&=
%\sup\left\{\frac{
% \lnorm
%\frac{m_0 (\cdot+2\Iml)}{m_\epsilon}\,h_\epsilon
% \rnorm_2^2
%+ \lnorm
% \frac{m_1 (\cdot+2\Iml)}{m_{1-\epsilon}}\,h_{1-\epsilon}
% \rnorm_2^2
%+2\theta
%\langle
%\frac{m_0 (\cdot+2\Iml)}{m_\epsilon}\,h_\epsilon
% ,
%\frac{m_1 (\cdot+2\Iml)}{m_{1-\epsilon}}\,h_{1-\epsilon}
%\rangle
%}{ \lnorm h_0
% \rnorm_2^2
%+ \lnorm h_1
% \rnorm_2^2
%+2\theta
%\langle h_0 \, , h_1
%\rangle
%}\,:\, h_k\in L^2(\R)\right\}
%\\
&=
\sup\left\{
\frac{
 \lnorm
h_\epsilon \,\quoz_{0,\epsilon}
 \rnorm_2^2
+ \lnorm
h_{1-\epsilon} \,\quoz_{1,\epsilon}
 \rnorm_2^2
+2 \theta\,\Re
\langle
h_\epsilon \, \quoz_{0,\epsilon} \, ,
h_{1-\epsilon} \, \quoz_{1,\epsilon}
\rangle_{L^2}
}{ \lnorm h_0
 \rnorm_2^2
+ \lnorm h_1
 \rnorm_2^2
+2\theta \Re
\langle h_0 \, , h_1
\rangle_{L^2}
}
\,:\,
%(h_0,h_1)\in L^2(\R)\oplus L^2(\R)\setminus\set{0}
\|h_0\|_2^2+\|h_1\|_2^2\neq 0\right\},
%\\
%&\frac{\int_{-\infty}^{+\infty} \labs \frac{m_0(u+2\Iml)}{m_0(u)} \fn h_0(u) \rabs^2 + \labs\frac{ m_1(u+2\Iml)}{m_1(u)} \, h_1(u) \rabs ^2 \,dt
%  + 2\theta \int_{-\infty}^{\infty} \Re\lpar \frac{m_0(u+2\Iml)}{m_0(u)} \fn h_0(u) \fn \frac{\bar m_1(u+2\Iml) }{ \bar m_1(u) }\fn \bar h_1(u)  \rpar \,dt }
%        {  \int_{-\infty}^{+\infty} \labs h_0(u) \rabs^2 + \labs h_1(u) \rabs ^2 \,dt
%  + 2\theta \int_{-\infty}^{\infty} \Re\lpar h_0(u) \fn \bar h_1(u) ) \rpar \,dt } \\
%&\frac{\int_{-\infty}^{+\infty} \labs \quoz_0(u) \fn h_0(u) \rabs^2 + \labs \quoz_1(u) \, h_1(u) \rabs ^2 \,dt
%  + 2\theta \int_{-\infty}^{\infty} \Re\lpar \quoz_0(u) \fn h_0(u) \fn \bar \quoz_1(u) \fn \bar h_1(u)  \rpar \,dt }
%        {  \int_{-\infty}^{+\infty} \labs h_0(u) \rabs^2 + \labs h_1(u) \rabs ^2 \,dt
%  + 2\theta \int_{-\infty}^{\infty} \Re\lpar h_0(u) \fn \bar h_1(u) ) \rpar \,dt }  \, ,
\end{aligned}
\]
where
\[
\quoz_{0,\epsilon}=\quoz_{0,\epsilon,\Rel+i\Iml}=\frac{m_{0,\Rel} (2\Iml -\cdot)}{m_{\epsilon,\Rel}}
\qquad \qquad
\quoz_{1,\epsilon}=\quoz_{1,\epsilon,\Rel+i\Iml}=-\frac{m_{1,\Rel} (2\Iml -\cdot)}{m_{1-\epsilon,\Rel}} \,.
\qedhere
\]
\end{proof}

\begin{corollary}\label{minimo}
For all $a,b \in \R^+$,
\[
\| \pi_{\la,\epsilon,a,b} \|_{\ub} \geq \| \pi_{\la,\epsilon} \|_{\ub} = \max\{ \lnorm \quoz_{0,\epsilon} \rnorm_\infty, \lnorm \quoz_{1,\epsilon} \rnorm_\infty \}.
\]
\end{corollary}

\begin{proof}
On the one hand, from the cases where $h_0=0$ or $h_1=0$ in~\eqref{norma_sigma}, we see that
\begin{equation}\label{min=}
  \| \pi_{\la,\epsilon,a,b} \|_{\ub} \geq  \max\{ \lnorm \quoz_{0,\epsilon} \rnorm_\infty, \lnorm \quoz_{1,\epsilon} \rnorm_\infty \} \qquad \forall a,b \in \R^+.
\end{equation}
On the other hand, taking $a=b$ in~\eqref{norma_sigma}, we deduce that
\[
\| \pi_{\la,\epsilon,a,a} \|_\ub \leq  \max\{ \lnorm \quoz_{0,\epsilon} \rnorm_\infty, \lnorm \quoz_{1,\epsilon} \rnorm_\infty \},
\]
and we have already observed that $\| \pi_{\la,\epsilon,a,a} \|_\ub=\| \pi_{\la,\epsilon} \|_\ub$.
\end{proof}

This corollary says that the Kunze--Stein representation~$\pi_{\la,\epsilon}$ has minimal norm in the sense explained in the introduction.
We are now going to show that when $a\neq b$ the inequality in Corollary~\ref{minimo} is strict.
To do this, we will use the  following general lemma.

\begin{lemma} \label{minimumnorm}
Consider the Hilbert space $L^2(\R) \oplus L^2(\R)$, equipped with the norm
\footnoteomit{To see that this is actually a norm, the shortest way is to
write $\lnorm (h_0, h_1) \rnorm_{\theta}^2 = \|h_0+\theta h_1\|_2^2+(1-\theta^2)\|h_1\|_2^2$.
}
%\[
%\lnorm (h_0, h_1) \rnorm_{\theta} = \lpar \int_{-\infty}^{+\infty} \labs h_0(u) \rabs^2 + \labs h_1(u) \rabs ^2 \,dt
%  + 2\theta \int_{-\infty}^{\infty} \Re\lpar h_0(u) \fn \bar h_1(u)  \rpar \,dt \rpar^{1/2} ,
%\]

\[
\lnorm (h_0, h_1) \rnorm_{\theta} =
{ \lnorm h_0
 \rnorm_2^2
+ \lnorm h_1
 \rnorm_2^2
+2\theta \Re
\langle h_0 \, , h_1
\rangle_{L^2}
},
\]
where $-1 < \theta < 1$, and the linear operator $T : (h_0, h_1) \mapsto (\quoz_0h_0, \quoz_1h_1)$ on this space, where $\quoz_0, \quoz_1 \in L^\infty(\R)$.
If $\theta = 0$, then $\lnorm T\rnorm_{\op} = \max\{ \lnorm \quoz_0 \rnorm_\infty, \lnorm \quoz_1 \rnorm_\infty \}$.
If $\theta \neq 0$ and $\lnorm \quoz_0 \quoz_1 \rnorm _\infty< \max\{ \lnorm \quoz_0 \rnorm_\infty, \lnorm \quoz_1 \rnorm_\infty \}^2$, then $\lnorm T \rnorm_{\op} > \max\{ \lnorm \quoz_0 \rnorm_\infty, \lnorm \quoz_1 \rnorm_\infty \}$.
\end{lemma}

\begin{proof}  
Without loss of generality, we suppose that  $\lnorm \quoz_0 \rnorm_\infty = \max\{ \lnorm \quoz_0 \rnorm_\infty, \lnorm \quoz_1 \rnorm_\infty \} = 1$. 
Taking  first  $h_0=0$ and then $h_1=0$  we see that $\lnorm T \rnorm_{\op}  \geq 1$.

When $\theta=0$, it is evident that $\lnorm T \rnorm_{\op}  = 1$.
Indeed,
\[
{ \lnorm \quoz_0\, h_0
 \rnorm_2^2
+ \lnorm \quoz_1\,h_1
 \rnorm_2^2
}
\leq \,  \lnorm h_0
 \rnorm_2^2
+ \lnorm h_1
 \rnorm_2^2,
\]
%\[
%\int_{-\infty}^{+\infty} (\labs \quoz_0(u) \fn h_0(u) \rabs^2 + \labs \quoz_1(u) \fn h_1(u) \rabs ^2) \,dt  \leq \int_{-\infty}^{+\infty} (\labs h_0(u) \rabs^2 + \labs h_1(u) \rabs ^2 )\,dt  ,
%\]
whence  $\lnorm T \rnorm_{\op}  \leq  1$.

%for all positive $\delta$, there is a nonnull measurable subset $E$ of $\R$ such that $\labs \quoz_0(u) \rabs > 1 - \delta$ for all $u$ in $E$, and so
%\[
%\lnorm T(\chi_{E},0) \rnorm_0 > (1-\delta) \lnorm (\chi_{E},0) \rnorm_0,
%\]
%whence $\lnorm T \rnorm > 1 - \delta$.
%

Suppose now $\theta \neq 0$ and  $\lnorm \quoz_0 \quoz_1 \rnorm_\infty = 1 - \epsilon$, where $\epsilon > 0$.  
We consider $T(\chi_{E(t)}, t\chi_{E(t)})$ when $t$ is very small, and $E(t)$ is chosen such that $\labs \quoz_0(u) \rabs > 1 - t^2$ for all $u$ in $E(t)$.
On the one hand,
\[
\begin{aligned}
\lnorm (\chi_{E(t)}, t\chi_{E(t)}) \rnorm_\theta
&= \lpar \int_{\R} \chi_{E(t)} (u) \,du + \int_{\R} t^2 \chi_{E(t)} (u) \,du +
2\theta \int_{\R} t \chi_{E(t)} (u) \,du \rpar^{1/2} \\
&= \lnorm \chi_{E(t)} \rnorm_2 (1 + 2 \theta t + t^2)^{1/2} \\
&= \lnorm \chi_{E(t)} \rnorm_2 (1 + \theta t + O(t^2)) ,
\end{aligned}
\]
so
\[
\frac{1}{ \lnorm (\chi_{E(t)}, t\chi_{E(t)}) \rnorm_\theta} = \frac{ 1 - \theta t + O(t^2) } { \lnorm \chi_{E(t)} \rnorm_2 } .
\]
On the other hand, $\lnorm (\quoz_0 \chi_{E(t)}, t \quoz_1\chi_{E(t)}) \rnorm_\theta$ is equal to
\[
\begin{aligned}
& \lpar \int_{\R} \labs \quoz_0(u)\rabs ^2 \fn\chi_{E(t)} (u) \,du + \int_{\R} t^2 \labs \quoz_1(u)\rabs ^2 \fn\chi_{E(t)} (u) \,du + 2\theta \Re \int_{\R} t \fn \quoz_0(u) \fn \bar \quoz_1(u) \fn \chi_{E(t)} (u) \,du \rpar^{1/2} .
% \\
%&\qquad = \lnorm \chi_{E(t)} \rnorm_2 \lpar 1 +  t \theta \frac { \int_{-\infty}^{\infty} t \fn \quoz_0(u) \fn \bar \quoz_1(u) \fn \chi_{E(t)} (u) \,du }{ \int_{-\infty}^{\infty} \labs \quoz_0\rabs ^2 \fn\chi_{E(t)} (u) \,du  }  + O(t^2) \rpar
\end{aligned}
\]
Now
\[
\int_{\R} \labs \quoz_0(u)\rabs ^2 \fn\chi_{E(t)} (u) \,du = (1 + O(t^2)) \lnorm \chi_{E(t)} \rnorm_2^2,
\]
 while
\[
\int_{\R} t^2 \labs \quoz_1(u)\rabs ^2 \fn\chi_{E(t)} (u) \,du = O(t^2) \lnorm \chi_{E(t)} \rnorm_2^2.
\]
It follows that
\begin{equation}\label{e:asterisk}
\frac{ \lnorm (\quoz_0 \chi_{E(t)}, t \quoz_1\chi_{E(t)}) \rnorm_\theta } { \lnorm (\chi_{E(t)}, t \chi_{E(t)}) \rnorm_\theta }
\geq 1 - \theta t + \theta t \,\frac{\ds\Re \int_{\R} \quoz_0(u) \fn \bar \quoz_1(u) \fn \chi_{E(t)} (u) \,du }{\ds \int_{\R} \chi_{E(t)} (u) \,du } +O(t^2) .
\end{equation}
Since
\[
\labs \frac{\ds \int_{\R} \quoz_0(u) \fn \bar \quoz_1(u) \fn \chi_{E(t)} (u) \,du }{\ds \int_{\R} \chi_{E(t)} (u) \,du }\rabs
\leq 1 - \epsilon ,
\]
when $t$ is very small and $t \theta < 0$, the right hand side of \eqref{e:asterisk} is strictly greater than $1$, so $\lnorm T \rnorm_{\op}  > 1$.
\end{proof}

In light of Lemma \ref{minimumnorm}, we need to analyse the functions $\quoz_{0,\epsilon}$, $\quoz_{1,\epsilon}$ defined in Lemma~\ref{moltiplicatore}.

\begin{lemma}\label{q0q1}
Suppose that $0<|\Rel|<\frac12$ and  $\epsilon+\Iml^2\neq 0$, and define $\quoz_{0,\epsilon}$ and $\quoz_{1,\epsilon}$ as in Lemma~\ref{moltiplicatore}.
Then
\[
\lnorm \quoz_{0,\epsilon}\, \quoz_{1,\epsilon} \rnorm_\infty
< \max\{ \lnorm \quoz_{0,\epsilon} \rnorm_\infty^2, \lnorm \quoz_{1,\epsilon} \rnorm_\infty^2 \}.
\]
\end{lemma}

\begin{proof}
Define $t_\Rel(u)=i\,\tan(\pi(\tfrac14+\tfrac{\Rel+iu}{2}))$.
Then by the reflection formula (see~\cite[p.~3, formula (6)]{Erd}),
\begin{align*}
m_{1,\Rel}(u)&=i\sqrt{2\pi}\,\, 2^{\Rel +iu} \,
\frac{\Gamma\left(\frac34+\frac{\Rel+iu}2\right)}{\Gamma\left(\frac34-\frac{\Rel+iu}2\right)}
\\
&=i\sqrt{2\pi}\,\, 2^{\Rel +iu} \,
\frac{\Gamma\left(\frac14+\frac{\Rel+iu}2\right)\, \sin\left(\pi\left(\frac34-\frac{\Rel+iu}2\right)\right)
}{
\Gamma\left(\frac14-\frac{\Rel+iu}2\right)\, \sin\left(\pi\left(\frac34+\frac{\Rel+iu}2\right)\right)
}
\\
%&=i\, m_{0,\Rel}(u)\,\tan(\pi(\tfrac14+\tfrac{\Rel+iu}{2}))
%:
&=m_{0,\Rel}(u)\,t_\Rel(u).
\end{align*}

We now prove the desired inequality. 
Note that
$
\lim_{u\to\infty}|\quoz_{0,\epsilon}(u)\,\quoz_{1,\epsilon}(u)|=1
$,
so
\[
\|\quoz_{0,\epsilon}\,\quoz_{1,\epsilon}\|_\infty \geq 1.
\]
Since $\quoz_{0,\epsilon}\,\quoz_{1,\epsilon}$ is a continuous function and $|\quoz_{0,\epsilon}\,\quoz_{1,\epsilon}(\Iml)|=1$, there exists $\bar u$ in $\R$ such that
\[
\|\quoz_{0,\epsilon}\,\quoz_{1,\epsilon}\|_\infty=|\quoz_{0,\epsilon}(\bar u)\quoz_{1,\epsilon}(\bar u)|.
\]
Suppose now that $\epsilon=1$ and write
\begin{align*}
|\quoz_{0,1}(\bar u)\quoz_{1,1}(\bar u)|
&=|\quoz_{0,1}(\bar u)|^2\, |t_\Rel(2\Iml-\bar u)\, t_\Rel(\bar u)|
\\
|\quoz_{0,1}(\bar u)\quoz_{1,1}(\bar u)|
&=|\quoz_{1,1}(\bar u)|^2\, \frac{1}{|t_\Rel(2\Iml-\bar u)\, t_\Rel(\bar u)|}.
\end{align*}
Since
$\displaystyle
|\tan(x+iy)|^2=\frac{\cosh(2y)-\cos(2x)}{\cosh(2y)+\cos(2x)},
$
we obtain that
\[
|t_\Rel(u)|^2=\frac{\cosh(\pi u)+\sin(\pi\Rel)}{\cosh(\pi u)-\sin(\pi\Rel)}\qquad \forall u\in\R ,
\]
therefore $|t_\Rel(u)|>1$ when $\Rel>0$ and $|t_\Rel(u)|<1$ when $\Rel<0$.
We easily conclude that
\[
\lnorm \quoz_{0,1}\, \quoz_{1,1} \rnorm_\infty
<
 \max\{ \lnorm \quoz_{0,1} \rnorm_\infty^2, \lnorm \quoz_{1,1} \rnorm_\infty^2 \}.
\]
In the case where $\epsilon=0$, we  write
\begin{align*}
\quoz_{0,0}(\bar u)\, \quoz_{1,0}(\bar u)
&=\left(\frac{m_{0,\Rel} (2\Iml-\bar u)}{m_{0,\Rel}(\bar u)}\right)^{\!2}\,
\frac{t_{\Rel} (2\Iml-\bar u)}{t_{\Rel}(\bar u)}
=\quoz_{0,0}^2(\bar u)\,\frac{t_{\Rel} (2\Iml-\bar u)}{t_{\Rel}(\bar u)}
\\
\quoz_{0,0}(\bar u)\, \quoz_{1,0}(\bar u)&=\left(\frac{m_{1,\Rel} (2\Iml-\bar u)}{m_{1,\Rel}(\bar u)}\right)^{\!2}\,
\frac{t_{\Rel} (\bar u)}{t_{\Rel}(2\Iml-\bar u)}
=\quoz_{1,0}^2(\bar u)\,\frac{t_{\Rel} (\bar u)}{t_{\Rel}(2\Iml-\bar u)} \,,
\end{align*}
and note that
\begin{align*}
\left|\frac{t_{\Rel} (2\Iml-u)}{t_{\Rel}(u)}\right|^2
&=\frac{\cosh(\pi(2\Iml-u))+\sin(\pi\Rel)}{\cosh(\pi(2\Iml-u))-\sin(\pi\Rel)}
\,\frac{\cosh(\pi u)-\sin(\pi\Rel)}{\cosh(\pi u)+\sin(\pi\Rel)}
\\
&=\frac{1+
\dfrac{2\sinh(\pi\Iml)\sinh(\pi(\Iml-u))}{\cosh(\pi u)+\sin(\pi\Rel)}
}{1+
\dfrac{2\sinh(\pi\Iml)\sinh(\pi(\Iml-u))}{\cosh(\pi u)-\sin(\pi\Rel)}
} \,.
\end{align*}
Suppose that $\Rel\in (0,\frac12)$ and $\Iml>0$, so that $\sin(\pi\Rel)>0$.
Then the fraction
\[
\left|\frac{t_{\Rel} (2\Iml-\bar u)}{t_{\Rel}(\bar u)}\right|
\]
is less than $1$ if $\Iml-\bar u>0$, is equal to $1$ if $\Iml-\bar u=0$, and is greater than $1$ if $\Iml-\bar u<0$.
Therefore if $\bar u\neq\Iml$,
\[
\lnorm \quoz_{0,0}\, \quoz_{1,0} \rnorm_\infty
<  \max\{ \lnorm \quoz_{0,0} \rnorm_\infty^2, \lnorm \quoz_{1,0} \rnorm_\infty^2 \}.
\]
Finally,
if $\bar u=\Iml$, we conclude that $\lnorm \quoz_{0,0}\, \quoz_{1,0} \rnorm_\infty=1$.
We claim that
\begin{equation*}\label{e:nonunit}
\|q_{0,0}\|_\infty>1,
\end{equation*}
from which the result follows easily.
To prove the claim, suppose that $|\quoz_{0,0}(u)|=1$ for all $u\in \R$. 
Then
\[
|m_0(2\Iml-u)|=|m_0(u)|=|m_0(-u)| \qquad \forall u\in \R ,
\]
that is, $|m_0|$ is $2\Iml$-periodic.
But if $\Rel>0$, $\lim_{u\to \infty}|m_0(u)|=\lim_{u\to \infty}\sqrt{2\pi}2^\Rel |u|^{\Rel}=+\infty$, which contradicts the periodicity of $|m_0|$.
Therefore there exists some $\tilde u$ such that $|\quoz_{0,0}(\tilde u)|\neq 1$.
Note that
\[
\quoz_{0,0}(u)=\frac{1}{\quoz_{0,0}(2\Iml -u)},
\]
therefore we conclude that
\[
\|\quoz_{0,0}\|_\infty\geq \max\{|\quoz_{0,0}(\tilde u)|, |\quoz_{0,0}(2\Iml-\tilde u)|\}
=\max\{|\quoz_{0,0}(\tilde u)|, 1/|\quoz_{0,0}(\tilde u)|\}>1.
\]
The cases where $\Iml<0$ or $\Rel<0$ may be treated similarly.
\end{proof}

\begin{corollary}\label{minstretto}
Suppose that  $a\not=b$, that $|\Rel|<\frac12$
and that $(\la,\epsilon)\neq (0,1)$.
Then
\begin{equation}\label{eq:strict-inequality}
\lnorm\pi_{\la,\epsilon,a,b} \rnorm_{\ub}>\lnorm\pi_{\la,\epsilon} \rnorm_{\ub}.
%= \max\{ \lnorm \quoz_{0,\epsilon} \rnorm_\infty, \lnorm \quoz_{1,\epsilon} \rnorm_\infty \}
\end{equation}
\end{corollary}

\begin{proof}
In the case where $\Rel\neq 0$ and $\epsilon+\Iml^2\neq 0$, the inequality follows from  Lemmas~\ref{moltiplicatore},  \ref{minimumnorm} and \ref{q0q1}.
In the remaining cases, the representation $\pi_{\la,\epsilon}$ is unitary and irreducible.
If equality holds in \eqref{eq:strict-inequality}, then $\pi_{\la,\epsilon,a,b}$ must also be unitary.
Then the identity map from $\Hilb_\Rel$ to $\Hilbab$, which is a similarity, is a multiple of a unitary operator, by Lemma~\ref{lemma:corol-of-Schur}, which implies that $a = b$.
%We can write
%\[
%\lnorm f \rnorm_{\Hilbab}^2=b\,\Hilbnorm{f}{\Rel}^2+(a-b)\Hilbnorm{f_+}{\Rel}^2
%\]
%where $f_+$ is the projection of $f$ on the space $\Hilb_\Rel^+$.
%Since $\pi_{\la,\epsilon}(w)$ does not preserve $\Hilb_\Rel^+$, we may find $f$ in $\Hilb_\Rel^+$ such that
%$\pi_{\la,\epsilon}(w)f$ is not in $\Hilb_\Rel^+$ so that
%\begin{align*}
%\lnorm\pi_{\la,\epsilon,a,b}(w)f\rnorm_\Hilbab^2
%&=b\,\Hilbnorm{\pi_{\la,\epsilon}(w)f}{\Rel}^2+(a-b)\Hilbnorm{(\pi_{\la,\epsilon}(w)f)_+}{\Rel}^2
%\\
%&=b\,\Hilbnorm{f}{\Rel}^2+(a-b)\Hilbnorm{(\pi_{\la,\epsilon}(w)f)_+}{\Rel}^2
%\\
%&=\lnorm f\rnorm_\Hilbab^2+(a-b)\Hilbnorm{(\pi_{\la,\epsilon}(w)f)_+}{\Rel}^2
%\end{align*}
%which proves that $\|\pi_{\la,\epsilon,a,b}\|_{\ub}>1$ when $a\neq b$.
\end{proof}

 %%%%%%%%%%%%%%%%%%

To complete the proof of Theorem~1.1, we find sharp estimates of $\lnorm \quoz_{0,\epsilon} \rnorm_\infty$ and $ \lnorm \quoz_{1,\epsilon} \rnorm_\infty $.
For this purpose we shall use the following technical lemma for the gamma function, in which we denote by $S$ the set $\lset z\in \C : \Re(z) \in [-\half,1] \rset$.

\begin{lemma} \label{Gamma}
%%%
%\footnote{ Funziona in generale per $\Re(z) \in [-\half,M] $, con $M>0$ e le costanti dipendenti da $M$.  }
%%%
There exist constants $C$ and $C'$ such that
\[
C \leq \frac{\labs(x+iy) \fn\Gamma(x + iy)\rabs}{e^{-\half \pi \labs y \rabs} (1 + \labs y \rabs)^{ x + \half}} \leq C'
\quad\forall x+iy \in S.
\]
\end{lemma}

\begin{proof}
%We leave this to the reader.
Take $x+iy \in S$.
Since $\Gamma(x-iy)=\overline{\Gamma(x+iy)}$,  we may suppose that $y>0$ and note that
\[
1< e^{y\arctan \frac{x+1}{y}}< e^{x+1}.
\]
By using Stirling's asymptotic expansion (see~\cite[p. 47, formula (2)]{Erd}),
\[
\Gamma(\zeta)=e^{-\zeta}\, e^{(\zeta-\half)\,\ln \zeta}\,\sqrt{2\pi} \left( 1+O(\zeta^{-1}) \right)
\qquad\zeta\to\infty,
\]
where $\zeta=x+1+iy$, we obtain
\[
\begin{aligned}
\left|
e^{-\zeta}\, e^{(\zeta-\frac12)\,\ln \zeta}
\right|
&=
\left|
e^{-(x+1)}\, e^{(x+\frac12)\,\ln |\zeta|-y\arctan \frac{y}{x+1}}\,
\right|
\\
&
=
\left|
e^{-(x+1)}\, |\zeta|^{x+\frac12}\,e^{ -\frac{\pi}2 y+y\arctan \frac{x+1}{y}}\,
\right|
\\
&\simeq
(1+y)^{x+\frac12}\,e^{ -\frac{\pi}2 y}.
\end{aligned}
\]
The result follows observing that $\Gamma(\zeta)=(x+iy) \fn\Gamma(x + iy)$.
\end{proof}

%\begin{proof}
%%We leave this to the reader.
%Take $x+iy \in S$.
%Since $\Gamma(x-iy)=\overline{\Gamma(x+iy)}$,  we may suppose that $y>0$.
%
%We use  Stirling's asymptotic expansion (see~\cite[p. 47, formula (2)]{Erd})
%\[
%\Gamma(\zeta)=e^{-\zeta}\, e^{(\zeta-\half)\,\ln \zeta}\,\sqrt{2\pi} \left( 1+O(\zeta^{-1}) \right)
%\]
%where $\zeta=x+1+iy$.
%
%Since
%\[
%1< e^{y\arctan \frac{x+1}{y}}< e^{x+1},
%\]
%then
%\[
%\begin{aligned}
%\left|
%e^{-\zeta}\, e^{(\zeta-\frac12)\,\ln \zeta}
%\right|
%&=
%\left|
%e^{-(x+1)}\, e^{(x+\frac12)\,\ln |\zeta|-y\arctan \frac{y}{x+1}}\,
%\right|
%\\
%&
%=
%\left|
%e^{-(x+1)}\, |\zeta|^{x+\frac12}\,e^{ -\frac{\pi}2 y+y\arctan \frac{x+1}{y}}\,
%\right|
%\\
%&\simeq
%(1+y)^{x+\frac12}\,e^{ -\frac{\pi}2 y}.
%\end{aligned}
%\]
% We obtain the result by observing that $\Gamma(\zeta)=(x+iy) \fn\Gamma(x + iy)$.
%\end{proof}

\begin{lemma}\label{stimen01}
Suppose that $\la =\Rel+i\Iml$, where $\Iml\not= 0$ and  $-\half <\Rel<\half$.
Then, with the constants $C$ and $C'$ of Lemma~\ref{Gamma},
\footnoteomit{
\[
\begin{gathered}
\left(\frac{C}{C'}\right)^2\,
\labs
\frac{\half-|\Rel|+2i\Iml}{\half+|\Rel|+2i\Iml}
\rabs
\,\frac{\half+|\Rel|}{\frac12-|\Rel |}
\leq \,
\frac{\| \quoz_{0,0,\lambda}\|_\infty}{\left(1+|\Iml |\right)^{|\Rel |}}
\,\leq
 \left( \frac{C'}{C}\right)^2\,\frac{1}{\frac12-|\Rel |}
;
\\
\left(\frac{C}{C'}\right)^2\,
\leq
(1+|\Iml|)^{-|\Rel|}\,\| \quoz_{1,0,\lambda}\|_\infty
\leq
2\, \left( \frac{C'}{C}\right)^2\,
;
\\
\left(\frac{C}{C'}\right)^2\,
\labs\frac{\half-\Rel+2i\Iml}{1+2i\Iml}\rabs
\frac{\frac14}{\half+\Rel}
\leq
(1+|\Iml|)^{-|\Rel|}\, \| \quoz_{0,1,\lambda}\|_\infty
\leq
2\, \left( \frac{C'}{C}\right)^2\,
\frac{1}{\half+\Rel};
\\
\left(\frac{C}{C'}\right)^2\,
\labs\frac{\half+\Rel+2i\Iml}{1+2i\Iml}\rabs
\frac{\frac14}{\half-\Rel}
\leq
(1+|\Iml|)^{-|\Rel|}\, \| \quoz_{1,1,\lambda}\|_\infty
\leq
2\, \left( \frac{C'}{C}\right)^2\,
\frac1{\half-\Rel}.
\end{gathered}
\]
}
\[
\begin{aligned}
\left(\frac{C}{C'}\right)^2\,
\labs
\frac{\half-|\Rel|+2i\Iml}{\half+|\Rel|+2i\Iml}
\rabs
\,\frac{\half+|\Rel|}{\frac12-|\Rel |}
\leq \,
&\frac{\| \quoz_{0,0,\lambda}\|_\infty}{\left(1+|\Iml |\right)^{|\Rel |}}
\,\leq
 \left( \frac{C'}{C}\right)^2\,\frac{1}{\frac12-|\Rel |}
\\ \\
\left(\frac{C}{C'}\right)^2\,
\leq \,
&\frac{\|\quoz_{1,0,\lambda}\|_\infty}{\left(1+|\Iml |\right)^{|\Rel |}}
\,\leq
2\, \left( \frac{C'}{C}\right)^2\,
\\ \\
\left(\frac{C}{C'}\right)^2\,
\labs\frac{\half-\Rel+2i\Iml}{1+2i\Iml}\rabs
\frac{\frac14}{\half+\Rel}
\leq \,
&\frac{\| \quoz_{0,1,\lambda}\|_\infty}{\left(1+|\Iml |\right)^{|\Rel |}}
\,\leq
2\, \left( \frac{C'}{C}\right)^2\,
\frac{1}{\half+\Rel}
\\  \\
\left(\frac{C}{C'}\right)^2\,
\labs\frac{\half+\Rel+2i\Iml}{1+2i\Iml}\rabs
\frac{\frac14}{\half-\Rel}
\leq \,
&\frac{\| \quoz_{1,1,\lambda}\|_\infty}{\left(1+|\Iml |\right)^{|\Rel |}}
\,\leq
2\, \left( \frac{C'}{C}\right)^2\,
\frac1{\half-\Rel}.
\\
\end{aligned}
\]
%%%
\footnoteomit{Or estimates of this type:
\[
\left(\frac{C}{C'}\right)^2\,
\frac{|\Iml|}{1+2|\Iml|}
\,\frac{1}{\frac12-|\Rel |}
\leq
\left(1+|\Iml |\right)^{-|\Rel |}\,\| \quoz_{0,0,\lambda}\|_\infty
\leq \left( \frac{C'}{C}\right)^2\,
\frac1{\half-|\Rel|};
\]
and analogous ones for the other norms. 
Estimates hold also when $\Iml=0$, but less significant.
}
%%%
\end{lemma}

\begin{proof}
By Lemma~\ref{Gamma},
\begin{align*}
\|\quoz_{0,0,\Rel+i\Iml}\|_\infty
&=\sup_{u\in \R}
\left|
\frac{
\Gamma
\left(\frac14+
\frac{\Rel+i(2\Iml-u)}{2}
\right)\,
\Gamma
\left(\frac14-
\frac{\Rel+iu}{2}
\right)
}%%%%%%%%%%%%%%% DIVISO
{
\Gamma
\left(\frac14-
\frac{\Rel+i(2\Iml-u)}{2}
\right)\,
\Gamma
\left(\frac14+
\frac{\Rel+iu}{2}
\right)
}
\right|
\\
&\leq
\left(\frac{C'}{C}\right)^2\,
\sup_{u\in \R}
\left|
\frac{1+|\Iml-u|}{1+|u|}
\right|^{\Rel}
\cdot\quad
\sup_{u\in \R}
\left|
\frac{\frac14-
\frac{\Rel}{2}-i(\Iml-u)}
{\frac14+
\frac{\Rel}{2}+i(\Iml-u)}
\right|
\left|
\frac{\frac14+
\frac{\Rel}{2}+iu}
{\frac14-
\frac{\Rel}{2}-iu}
\right|
\\
&
\leq
\left(\frac{C'}{C}\right)^2\,
(1+|\Iml|)^{|\Rel|}\,
\left(\sup_{u\in \R}
 \frac
 {\big(\frac14-\frac{\Rel}2\big)^2+(\Iml-u)^2}
 {\big(\frac14+\frac{\Rel}2\big)^2+(\Iml-u)^2}
 \cdot
 \sup_{u\in \R}\frac
 {\big(\frac14+\frac{\Rel}2\big)^2+u^2}
 {\big(\frac14-\frac{\Rel}2\big)^2+u^2}
\right)^{1/2}
\\
&=
\left(\frac{C'}{C}\right)^2\,
(1+|\Iml|)^{|\Rel|}\, \frac1{\half-|\Rel|} .
\end{align*}

On the other hand,
\begin{align*}
\|\quoz_{0,0,\Rel+i\Iml}\|_\infty
&\geq
\max\{| \quoz_{0,0,\Rel+i\Iml}(0) |, | \quoz_{0,0,\Rel+i\Iml}(2\Iml) |\}
\\
&\geq \labs
\frac{\Gamma\left(  \frac14 +\frac{|\Rel|}2  +i\Iml  \right)}
{\Gamma\left(  \frac14 -\frac{|\Rel|}2  -i\Iml  \right)}
\cdot
\frac{\Gamma\left(  \frac14 -\frac{|\Rel|}2    \right)}
{\Gamma\left(  \frac14 +\frac{|\Rel|}2    \right)}
\rabs
\\
&\geq \left(\frac{C}{C'}\right)^2\, \left(1+|\Iml |\right)^{|\Rel|}\,
\labs
\frac
{
  \frac14 -\frac{|\Rel|}2+i\Iml
  }
{   \frac14 +\frac{|\Rel|}2+i\Iml
}
\rabs
\cdot
\labs
\frac{\frac14 +\frac{|\Rel|}2}
{\frac14 -\frac{|\Rel|}2 }\rabs ,
%\\
%&\geq
%\left(\frac{C}{C'}\right)^2\, \left(1+|\Iml |\right)^{|\Rel|}\,
%\frac{|\Iml|}{1+2|\Iml|}
%\,\frac{1}{\frac12-|\Rel|},
\end{align*}
concluding the proof of the estimate for $\|\quoz_{0,0,\Rel+i\Iml}\|_\infty$.
%%%%

The other computations are similar and we leave the details to the reader.
\footnoteomit{
%%%%%%
%\begin{align*}
%\|\quoz_{0,0,\Rel+i\Iml}\|_\infty
%&\geq
%| \quoz_{0,0,\Rel+i\Iml}(0) |
%=\labs
%\frac{\Gamma\left(  \frac14 +\frac{\Rel}2  +i\Iml  \right)}
%{\Gamma\left(  \frac14 -\frac{\Rel}2  -i\Iml  \right)}
%\cdot
%\frac{\Gamma\left(  \frac14 -\frac{\Rel}2    \right)}
%{\Gamma\left(  \frac14 +\frac{\Rel}2    \right)}
%\rabs
%\\
%&\geq \left(\frac{C}{C'}\right)^2\, \left(1+|\Iml |\right)^{\Rel}\,
%\labs
%\frac
%{
%  \frac14 -\frac{\Rel}2+i\Iml
%  }
%{   \frac14 +\frac{\Rel}2+i\Iml
%}
%\rabs
%\cdot
%\labs
%\frac{\frac14 +\frac{\Rel}2}
%{\frac14 -\frac{\Rel}2 }\rabs
%\\
%&\geq
%\left(\frac{C}{C'}\right)^2\, \left(1+|\Iml |\right)^{\Rel}\,
%\frac{|\Iml|}{1+2|\Iml|}
%\,\frac{1}{\frac12-\Rel}
%\end{align*}
%
%%%%%
%
%
% The case where $-1/2<\Rel<0$ is analogous:
%\begin{align*}
%\|\quoz_{0,0,\Rel+i\Iml}\|_\infty
%&\geq
%| \quoz_{0,\Rel+i\Iml}(-2\Iml) |
%=\labs
%\frac{\Gamma\left(  \frac14 -\frac{\Rel}2  +i\Iml  \right)}
%{\Gamma\left(  \frac14 +\frac{\Rel}2  +i\Iml  \right)}
%\cdot
%\frac{\Gamma\left(  \frac14 +\frac{\Rel}2    \right)}
%{\Gamma\left(  \frac14 -\frac{\Rel}2    \right)}
%\rabs
%\\
%&\geq
%\left(\frac{C}{C'}\right)^2\, \left(1+|\Iml |\right)^{-\Rel }\,
%\frac{|\Iml|}{1+2|\Iml|}
%\,\frac{1}{\frac12+\Rel}
%.
%\end{align*}
%%%%%%%%%%%%%%%%

Reasoning as before we obtain
\begin{align*}
\|\quoz_{1,0,\Rel+i\Iml}\|_\infty
&=\sup_{u\in \R}
\left|
\frac{
\Gamma
\left(\frac34+
\frac{\Rel+i(u+2\Iml)}{2}
\right)\,
\Gamma
\left(\frac34-
\frac{\Rel+iu}{2}
\right)
}%%%%%%%%%%%%%%% DIVISO
{
\Gamma
\left(\frac34-
\frac{\Rel+i(u+2\Iml)}{2}
\right)\,
\Gamma
\left(\frac34+
\frac{\Rel+iu}{2}
\right)
}
\right|
\\
&\leq
\left(\frac{C'}{C}\right)^2\,
\sup_{u\in \R}
\left|
\frac{1+|u+\Iml|}{1+|u|}
\right|^{\Rel}
\cdot\quad
\sup_{u\in \R}
\left|
\frac{\frac34-
\frac{\Rel}{2}+i(u+\Iml)}
{\frac34+
\frac{\Rel}{2}+i(u+\Iml)}
\right|
\left|
\frac{\frac34+
\frac{\Rel}{2}+iu}
{\frac34-
\frac{\Rel}{2}+iu}
\right|
\\
&
=\left(\frac{C'}{C}\right)^2\,
(1+|\Iml|)^{|\Rel|}\,
\cdot\quad
\sup_{u\in \R}
\left(
 \frac
 {\big(\frac34-\frac{\Rel}2\big)^2+(u+\Iml)^2}
 {\big(\frac34+\frac{\Rel}2\big)^2+(u+\Iml)^2}
 \cdot
 \frac
 {\big(\frac34+\frac{\Rel}2\big)^2+u^2}
 {\big(\frac34-\frac{\Rel}2\big)^2+u^2}
\right)^{1/2}
\\
&\leq 2
\left(\frac{C'}{C}\right)^2\,
(1+|\Iml|)^{|\Rel|}\,
\end{align*}
and
%%%%
\begin{align*}
\|\quoz_{1,0,\Rel+i\Iml}\|_\infty
&\geq
\max\{| \quoz_{1,0,\Rel+i\Iml}(0) |, | \quoz_{1,0,\Rel+i\Iml}(-2\Iml) |\}
\geq
\labs
\frac{\Gamma\left(  \frac34 +\frac{|\Rel|}2  +i\Iml  \right)}
{\Gamma\left(  \frac34 -\frac{|\Rel|}2  -i\Iml  \right)}
\cdot
\frac{\Gamma\left(  \frac34 -\frac{|\Rel|}2    \right)}
{\Gamma\left(  \frac34 +\frac{|\Rel|}2    \right)}
\rabs
\\
&\geq
\left(\frac{C}{C'}\right)^2\, \left(1+|\Iml |\right)^{|\Rel|}\,
\labs
\frac
{
  \frac34 -\frac{|\Rel|}2+i\Iml
  }
{   \frac34 +\frac{|\Rel|}2+i\Iml
}
\rabs
\cdot
\labs
\frac{\frac34 +\frac{|\Rel|}2}
{\frac34 -\frac{|\Rel|}2 }\rabs
\\
&\geq
\left(\frac{C}{C'}\right)^2\, \left(1+|\Iml |\right)^{|\Rel|}\,.
%\frac{|\Iml|}{1+|\Iml|}
\end{align*}

%%%%%%
%\begin{align*}
%\|\quoz_{1,0,\Rel+i\Iml}\|_\infty
%&=\sup_{u\in \R}
%\left|
%\frac{
%\Gamma
%\left(\frac34+
%\frac{\Rel+i(u+2\Iml)}{2}
%\right)\,
%\Gamma
%\left(\frac34-
%\frac{\Rel+iu}{2}
%\right)
%}%%%%%%%%%%%%%%% DIVISO
%{
%\Gamma
%\left(\frac34-
%\frac{\Rel+i(u+2\Iml)}{2}
%\right)\,
%\Gamma
%\left(\frac34+
%\frac{\Rel+iu}{2}
%\right)
%}
%\right|
%\\
%&\geq
%\left(\frac{C}{C'}\right)^2\,
%\sup_{u\in \R}
%\left|
%\frac{1+|u+\Iml|}{1+|u|}
%\right|^{\Rel}
%\cdot\quad
%\sup_{u\in \R}
%\left|
%\frac{\frac34-
%\frac{\Rel}{2}+i(u+\Iml)}
%{\frac34+
%\frac{\Rel}{2}+i(u+\Iml)}
%\right|
%\left|
%\frac{\frac34+
%\frac{\Rel}{2}+iu}
%{\frac34-
%\frac{\Rel}{2}+iu}
%\right|
%\\
%&\geq \frac12
%\left(\frac{C}{C'}\right)^2\,
%(1+|\Iml|)^{|\Rel|}\, .
%\end{align*}

In the case where $\epsilon=1$,
\begin{align*}
\|\quoz_{0,1,\Rel+i\Iml}\|_\infty
&=\sup_{u\in \R}
\left|
\frac{
\Gamma
\left(\frac14+
\frac{\Rel+i(u+2\Iml)}{2}
\right)\,
\Gamma
\left(\frac34-
\frac{\Rel+iu}{2}
\right)
}%%%%%%%%%%%%%%% DIVISO
{
\Gamma
\left(\frac14-
\frac{\Rel+i(u+2\Iml)}{2}
\right)\,
\Gamma
\left(\frac34+
\frac{\Rel+iu}{2}
\right)
}
\right|
\\
&\leq
\left(\frac{C'}{C}\right)^2\,
\sup_{u\in \R}
\left|
\frac{1+|u+\Iml|}{1+|u|}
\right|^{\Rel}
\cdot\quad
\sup_{u\in \R}
\left|
\frac{\frac14-
\frac{\Rel}{2}+i(u+\Iml)}
{\frac14+
\frac{\Rel}{2}+i(u+\Iml)}
\right|
\left|
\frac{\frac34+
\frac{\Rel}{2}+iu}
{\frac34-
\frac{\Rel}{2}+iu}
\right|
 .
\end{align*}
%%%
Therefore, if  $0<\Rel<\half$
\[
\|\quoz_{0,1,\Rel+i\Iml}\|_\infty
\leq
\left(\frac{C'}{C}\right)^2\,
(1+|\Iml|)^{|\Rel|}\,
\left|
\frac{\frac34+
\frac{\Rel}{2} }
{\frac34-
\frac{\Rel}{2}}
\right|
\leq  2\,
\left(\frac{C'}{C}\right)^2\,
(1+|\Iml|)^{|\Rel|}\,
\]
and
\[
\begin{aligned}
\|\quoz_{0,1,\Rel+i\Iml}\|_\infty &
\geq |\quoz_{0,1,\Rel+i\Iml}(0)|
=\left|
\frac{
\Gamma
\left(\frac14+
\frac{\Rel}{2}+i\Iml
\right)\,
\Gamma
\left(\frac34-
\frac{\Rel}{2}
\right)
}%%%%%%%%%%%%%%% DIVISO
{
\Gamma
\left(\frac14-
\frac{\Rel}{2}+i \Iml
\right)\,
\Gamma
\left(\frac34+
\frac{\Rel}{2}
\right)
}
\right|
\\
&\geq
\left(\frac{C}{C'}\right)^2\,
(1+|\Iml|)^{|\Rel|}\, \labs\frac{\half-\Rel+2i\Iml}{\half+\Rel+2i\Iml}\rabs
\\
&\geq
\left(\frac{C}{C'}\right)^2\,
(1+|\Iml|)^{|\Rel|}\, \labs\frac{\half-\Rel+2i\Iml}{1+2i\Iml}\rabs
.
\end{aligned}
\]
If $-\half<\Rel <0$,
\]
\|\quoz_{0,1,\Rel+i\Iml}\|_\infty
\leq
\left(\frac{C'}{C}\right)^2\,
(1+|\Iml|)^{|\Rel|}\,
\left|
\frac{\frac14-
\frac{\Rel}{2} }
{\frac14+
\frac{\Rel}{2}}
\right|
\leq
\left(\frac{C'}{C}\right)^2\,
(1+|\Iml|)^{|\Rel|}\,
\frac1{{\half+\Rel }}
\[
and
\[
\begin{aligned}
\|\quoz_{0,1,\Rel+i\Iml}\|_\infty
&\geq |\quoz_{0,1,\Rel+i\Iml}(-2\Iml)|
=\left|
\frac{
\Gamma
\left(\frac14+
\frac{\Rel}{2}
\right)\,
\Gamma
\left(\frac34-
\frac{\Rel}{2}+i\Iml
\right)
}%%%%%%%%%%%%%%% DIVISO
{
\Gamma
\left(\frac14-
\frac{\Rel}{2}
\right)\,
\Gamma
\left(\frac34+
\frac{\Rel}{2}+i \Iml
\right)
}
\right|
\\
&\geq
\left(\frac{C}{C'}\right)^2\,
(1+|\Iml|)^{|\Rel|}\,
\left|
\frac{\frac34+
\frac{\Rel}{2}+i\Iml}
{\frac34-
\frac{\Rel}{2}+i\Iml}
\right|
\left|
\frac{\frac14-
\frac{\Rel}{2}}
{\frac14+
\frac{\Rel}{2}}
\right|
\\
&\geq
\left(\frac{C}{C'}\right)^2\,
(1+|\Iml|)^{|\Rel|}\,
\frac{\frac14}
{\frac12+
{\Rel}}.
\end{aligned}
\]
Finally,
\[
\begin{aligned}
\|\quoz_{1,1,\Rel+i\Iml}\|_\infty
&=
\sup_{u\in \R}
\labs
\frac{
\Gamma(\frac34+\frac{\Rel+i(u+2\Iml)}{2})}{\Gamma(\frac34-\frac{\Rel+i(u+2\Iml)}{2})
}
\frac{\Gamma(\frac14-\frac{\Rel+iu}{2})}{\Gamma(\frac14+\frac{\Rel+iu}{2})}
\rabs
\\
&\leq
\left(\frac{C'}{C}\right)^2
\sup_{u\in \R}
\left|
\frac{1+|u+\Iml|}{1+|u|}
\right|^{\Rel}
\cdot\quad
\sup_{u\in \R}
\left|
\frac{\frac34-
\frac{\Rel}{2}+i(u+\Iml)}
{\frac34+
\frac{\Rel}{2}+i(u+\Iml)}
\right|
\left|
\frac{\frac14+
\frac{\Rel}{2}+iu}
{\frac14-
\frac{\Rel}{2}+iu}
\right|
 .
\end{aligned}
\]
Therefore when $0<\Rel<\half$
\[
\|\quoz_{1,1,\Rel+i\Iml}\|_\infty
\leq
\left(\frac{C'}{C}\right)^2
\left(
1+|\Iml|
\right)^{|\Rel|}\,
\frac{1}{\half-\Rel}
\]
and
\[
\begin{aligned}
\|\quoz_{1,1,\Rel+i\Iml}\|_\infty
&\geq |\quoz_{1,1,\Rel+i\Iml}(0)| =
\labs
\frac{
\Gamma(\frac34+\frac{\Rel}{2}+i\Iml)}{\Gamma(\frac34-\frac{\Rel}{2}-i\Iml)
}
\frac{\Gamma(\frac14-\frac{\Rel}{2})}{\Gamma(\frac14+\frac{\Rel}{2})}
\rabs
\\
&\geq
\left(\frac{C}{C'}\right)^2
\left(
1+|\Iml|
\right)^{|\Rel|}\,
\labs
\frac{
\frac34-\frac{\Rel}{2}+i\Iml}{\frac34+\frac{\Rel}{2}+i\Iml
}\rabs\labs
\frac{\frac14+\frac{\Rel}{2}}{\frac14-\frac{\Rel}{2}}
\rabs
\\
&\geq
\left(\frac{C}{C'}\right)^2
\left(
1+|\Iml|
\right)^{|\Rel|}\,\frac{\frac14}{\half-\Rel}.
\end{aligned}
\]
When $-\half<\Rel<0$,
\[
\|\quoz_{1,1,\Rel+i\Iml}\|_\infty
\leq
\left(\frac{C'}{C}\right)^2
\left(
1+|\Iml|
\right)^{|\Rel|}\,
\frac{\frac34-\frac{\Rel}{2}}{\frac34+\frac{\Rel}{2}}
\leq 2
\left(\frac{C'}{C}\right)^2
\left(
1+|\Iml|
\right)^{|\Rel|}
\]
and
\[
\begin{aligned}
\|\quoz_{1,1,\Rel+i\Iml}\|_\infty
&\geq |\quoz_{1,1,\Rel+i\Iml}(-2\Iml)| =
\labs
\frac{
\Gamma(\frac34+\frac{\Rel}{2})}{\Gamma(\frac34-\frac{\Rel}{2})
}
\frac{\Gamma(\frac14-\frac{\Rel}{2}+i\Iml)}{\Gamma(\frac14+\frac{\Rel}{2}+i\Iml)}
\rabs
\\
&\geq
\left(\frac{C}{C'}\right)^2
\left(
1+|\Iml|
\right)^{|\Rel|}\,
\labs
\frac{
\frac34-\frac{\Rel}{2}}{\frac34+\frac{\Rel}{2}
}\rabs\labs
\frac{\frac14+\frac{\Rel}{2}+i\Iml}{\frac14-\frac{\Rel}{2}+i\Iml}
\rabs
\\
&\geq
\left(\frac{C}{C'}\right)^2
\left(
1+|\Iml|
\right)^{|\Rel|}\,\labs
\frac{\frac12+\Rel+2i\Iml}{1+2i\Iml}
\rabs.
\end{aligned}
\]
This concludes the proof.
}
\end{proof}

\begin{corollary}\label{stimaKS}
Suppose that $\la =\Rel+i\Iml$, where $-\half <\Rel<\half$.
Then when $\Iml$ is large,
\[
\lnorm \pi_{\la,\epsilon}\rnorm_{\ub}
\simeq \frac{(1+|\Iml|)^{|\Rel|}}{\half-|\Rel |}.
\]
\end{corollary}

\begin{proof}
By Corollary~\ref{minimo},
\[
\lnorm \pi_{\la,\epsilon}\rnorm_{\ub}=\max\{ \lnorm \quoz_{0,\epsilon} \rnorm_\infty, \lnorm \quoz_{1,\epsilon} \rnorm_\infty \},
\]
and our estimates follow easily.
\end{proof}

We conclude with the observation that some parts of the proof of Theorem~1.1 generalize easily to $SO(n,1)$, with $n>2$, and other parts with more difficulty.
We believe that everything may be extended, with different formulae, and plan to return to this soon.

\end{document}